\documentclass[12pt]{article}
\usepackage{amsthm,amsmath,amssymb,natbib,enumerate,a4wide,color}

\usepackage[colorlinks]{hyperref}

\newtheorem{theorem}{Theorem}[section]

\newtheorem{assumption}{Assumption}[section]
\newtheorem{lemma}{Lemma}[section]

\newtheorem{corollary}{Corollary}[section]

\theoremstyle{remark}
\newtheorem{remark}{Remark}[section]
\newtheorem{example}{Example}[section]

\newcommand{\esp}{{\mathbb E}}
\newcommand{\pr}{{\mathbb P}}
\newcommand{\var}{{\mathrm{var}}}
\newcommand{\cov}{{\mathrm{cov}}}

\newcommand{\tcr}{\textcolor{red}}

\newcommand{\plim}{{\ \stackrel{\mathbb{P}}\to} \ }
\newcommand{\equallaw}{{\ \stackrel{law}=} \ }
\newcommand{\dlim}{{\to}}
\newcommand{\func}{{\Rightarrow}}

\numberwithin{equation}{section}

\begin{document}

\title{Limit Laws in Transaction-Level Asset Price Models}

\author{Alexander Aue\thanks{Department of Statistics, University of California,
    Davis, One Shields Avenue, Davis, CA 95616, USA, email: {\tt
      alexaue@wald.ucdavis.edu}} \and Clifford Hurvich\footnote{Stern School of
    Business, New York University, Henry Kaufman Management Center, 44 West
    Fourth Street, New York, NY 10012, USA, email: {\tt churvich@stern.nyu.edu}}
  \and Philippe Soulier\footnote{Universit\'e Paris X, 200 avenue de la
    R\'epublique, 92001 Nanterre cedex, France, email: {\tt
      philippe.soulier@u-paris10.fr}}}

\maketitle

\begin{abstract}
  We consider pure-jump transaction-level models for asset prices in continuous
  time, driven by point processes. In a bivariate model that admits
  cointegration, we allow for time deformations to account for such effects as
  intraday seasonal patterns in volatility, and non-trading periods that may be
  different for the two assets. We also allow for asymmetries (leverage
  effects). We obtain the asymptotic distribution of the log-price process. We
  also obtain the asymptotic distribution of the ordinary least-squares
  estimator of the cointegrating parameter based on data sampled from an
  equally-spaced discretization of calendar time, in the case of weak fractional
  cointegration. For this same case, we obtain the asymptotic distribution for a
  tapered estimator under more general assumptions.  In the strong fractional
  cointegration case, we obtain the limiting distribution of a
  continuously-averaged tapered estimator as well as other estimators of the
  cointegrating parameter, and find that the rate of convergence can be affected
  by properties of intertrade durations. In particular, the persistence of
  durations (hence of volatility) can affect the degree of cointegration. We
  also obtain the rate of convergence of several estimators of the cointegrating
  parameter in the standard cointegration case. Finally, we consider the
  properties of the ordinary least squares estimator of the regression parameter
  in a spurious regression, i.e., in the absence of cointegration.

\end{abstract}

\section{Introduction}
\label{sec:intro}
The increasingly widespread availability of transaction-level financial price
data motivates the development of models to describe such data, as well as
theory for widely-used statistics of interest under the assumption of a given
transaction-level generating mechanism. We focus here on a bivariate pure-jump
model in continuous time for log prices proposed by
\cite{hurvich:wang:2009preprint,hurvich:wang:2009} which yields fractional
or standard cointegration. The motivation for using a pure-jump model is that
observed price series are step functions, since no change is possible in
observed prices during time periods when there are no transactions. Examples of
data sets that would fit into the framework of this model include: buy prices
and sell prices of a single stock; prices of two different stocks within the
same industry; stock and option prices of a given company; option prices on a
given stock with different degrees of maturity or moneyness; corporate bond
prices at different maturities for a given company; Treasury bond prices at
different maturities.

Though our paper is not entirely focused on the case of fractional
cointegration, we present here some evidence that this case may arise
in practice in financial econometrics. We considered option and
underlying best-available bid prices for 69 different options on IBM
at 390 one-minute intervals from 9:30 AM to 4 PM on May 31, 2007.
Using a log-periodogram estimator based on $390^{0.5}$ frequencies,
we found that the logs of the original series had estimated memory
parameters close to 1, while the residuals from the OLS regression of
the log stock price on the log option price had estimated memory
parameters that were typically less than 1. Specifically, of the 69
estimated memory parameters based on these residuals, the values
ranged from 0.05 to 1.14 with a mean of 0.55 and a standard deviation
of 0.28, with 30 of these estimates lying between 0.5 and 1, while 32
were between 0 and 0.5. Thus, there is evidence for cointegration in
most of the series studied, and often the evidence points towards
fractional rather than standard cointegration. Furthermore, the OLS
estimate of the cointegrating parameter (assuming that cointegration
exists) ranged from $-0.21$ to 0.39, with a mean of 0.04 and a
standard deviation of 0.13. This provides evidence that the
cointegrating parameter is in general not equal to one in the present
context, so it is of interest to study properties of estimates of
this parameter.

Two basic questions that we address in this paper are the asymptotic
distribution of the log prices as time $t \to \infty$, and of the usual
OLS estimator of the cointegrating parameter based on $n$ observations of the
log prices at equally-spaced time intervals as $n \to \infty$. Most of
the existing methods for deriving such limit laws (see \cite{MR1869235}) cannot
be applied here because the continuous-time log-price series are not diffusions
and because the discretized log-price series are not linear in either an
i.i.d.~sequence, a martingale difference sequence or a strong mixing
sequence. Nevertheless, it is of interest to know whether and under what
conditions the existing limit laws, based, say, on linearity assumptions in
discrete time, continue to hold under a transaction-level generating mechanism.

In the model of \cite{hurvich:wang:2009,hurvich:wang:2009preprint} the price
process in continuous time is specified by a counting process giving the
cumulative number of transactions up to time $t$, together with the process of
changes in log price at the transaction times. This structure corresponds to the
fact that most transaction-level data consists of a time stamp giving the
transaction time as well as a price at that time. In such a setting, another
observable quantity of interest is the $durations$, i.e., the waiting times
between successive transactions of a given asset. There is a growing literature
on univariate models for durations, including the seminal paper of
\cite{MR1639411} on the autoregressive conditional duration models (ACD), as
well as \cite{MR2057105} on the stochastic duration model (SCD),
and \cite{Deo20103715} on the long-memory stochastic duration model (LMSD).

\cite{deo:hurvich:soulier:wang:2009} showed that, subject to regularity
conditions, if partial sums of centered durations, scaled by $n^{-(d+1/2)}$ with
$d \in [0,1/2)$, satisfy a functional central limit theorem then the counting
process $N(t)$ has long or short memory (for $d>0$, $d=0$, respectively) in the
sense that ${\rm Var} N(t) \sim C t^{2d+1}$ as $t \to \infty$ (with
$C>0$), and they gave conditions under which this scaling would lead to long
memory in volatility. In particular, LMSD durations with $d>0$ lead to long
memory in volatility. The latter property has been widely observed in the
econometrics literature, while evidence for long memory in durations was found
in \cite{Deo20103715}.

\cite{hurvich:wang:2009,hurvich:wang:2009preprint} did not derive limit laws for
the log price series or the OLS estimator of the cointegrating parameter, but
focused instead on properties of variances and covariances for log price series
and returns, and on lower bounds on the rate of convergence for the OLS
estimator.

In this paper, for a slightly modified version of the model of
\cite{hurvich:wang:2009,hurvich:wang:2009preprint}, but under assumptions that
are more general than theirs, we obtain the limit law for the log prices, and
for the OLS and tapered estimators of the cointegrating parameter. In our result
on the limit law for log prices, Theorem \ref{theo:fclt-levels}, we allow for a
stochastic time-varying intensity function in the counting processes. This
allows for such effects as dynamic intraday seasonality in volatility (as
observed, for example, in \cite{MR2275995}, as well as fixed nontrading
intervals such as holidays and overnight periods. We also allow in most of our
results for asymmetries (leverage effects), and show that this opens up the
possibility that long memory in durations may affect the rate of convergence of
estimators of the cointegrating parameter. This raises some heretofore
unrecognized ambiguities in the choice of a definition of standard
cointegration. Finally, we consider the properties of the ordinary least squares
estimator in a spurious regression, i.e., in the absence of
cointegration.

The remainder of this paper is organized as follows. In
Section~\ref{sec:model-assumptions} we write the model for the log price series
and state our assumptions on the counting process, the time-deformation
functions, and the return shocks. In Section \ref{sec:main}, we provide our main
results on: the long-run behavior of the log-price process (Subsection
\ref{sec:LongRun}), the OLS estimator for the cointegrating parameter under weak
fractional, strong fractional and standard cointegration
(Subsection~\ref{sec:ols-cointegration}), a tapered estimator under weak
fractional, strong fractional and standard cointegration
(Subsection~\ref{sec:Taper}), a continuously-averaged tapered estimator under
strong fractional and standard cointegration (Subsection
\ref{sec:Taper-continuous-time}) and the ordinary least squares estimator in the
spurious regression case (Subsection \ref{sec:spurious}). Section
\ref{sec:proofs} provides proofs.

\section{Transaction level model}
\label{sec:model-assumptions}

As in \cite{hurvich:wang:2009,hurvich:wang:2009preprint}, we consider a
bivariate pure-jump transaction-level price model that enables cointegration.
We define the log-price process $y=(y_1,y_2)=(y(t)\colon t\geq 0)$ by
\begin{align}
  y_1(t) & = \sum_{k=1}^{N_1(t)}(e_{1,k} + \eta_{1,k}) +
  \theta  \sum_{k=1}^{N_2(t_{1,N_1(t)})}   e_{2,k}, \label{eq:P1}  \\
  y_2(t) & = \sum_{k=1}^{N_2(t)}(e_{2,k} + \eta_{2,k}) + \theta^{-1}
  \sum_{k=1}^{N_1(t_{2,N_2(t)})} e_{1,k}, \label{eq:P2}
\end{align}
where for $i=1,2$, $N_i(\cdot)$ are counting processes on the real line (see
\citet[page 43]{daley:vere-jones:2003}) such that, for $t \geq 0$, $N_i(t)$
gives the total number of transactions of Asset $i$ in $(0,t]$, and $t_{i,k}$ is
the clock time (calendar time) for the $k$th transaction of Asset $i$, with
$\cdots t_{i,-1} \leq t_{i,0} \leq 0 < t_{i,1} \leq t_{i,2} \cdots$.  The
quantity $N_2(t_{1,N_1(t)})$ denotes the number of transactions of Asset 2
between time $0$ and the time $t_{1,N_1(t)}$ of the most recent transaction of
Asset 1, with an analogous interpretation for $N_1(t_{2,N_2(t)})$. The efficient
shock sequences $\{e_{i,k}\}_{k=1}^{\infty}$ model the permanent component and
the microstructure noise sequences $\{\eta_{i,k}\}_{k=1}^{\infty}$ model the
transitory component of the log-price process. Efficient shock spillover effects
are weighted by $\theta$ and $\theta^{-1}$, thus yielding cointegration with
cointegrating parameter $\theta$, assumed nonzero. A detailed economic
justification for this model, derivation of a common-components representation,
as well as a comparison with certain discrete-time models, is given in
\cite{hurvich:wang:2009,hurvich:wang:2009preprint}.  \tcr{more explanations}

In the mathematical theory presented in this paper, all random variables and
stochastic processes are defined on a single probability space
$(\Omega,\mathcal{F},\pr)$. Expectation with respect to $\pr$ will be denoted by
$\esp$ and 
$\var$ and $\cov$ will denote the variance and covariance with respect to
$\pr$. Convergence in $\pr$-probability will be denoted by~$\plim$, convergence
in distribution under $\pr$ of sequences of random variables will be denoted
by~$\dlim$.  We use $\func$ to denote weak convergence under $\pr$ in the space
$\mathcal D([0,\infty))$ of left-limited, right-continuous (c\`adl\`ag)
functions, endowed with Skorohod's $J_1$ topology.  See \cite{MR0233396} or
\cite{MR1876437} for details about weak convergence in $\mathcal
D([0,\infty))$. Whenever the limiting process is continuous, this topology can
be replaced by the topology of uniform convergence on compact sets.

Following \citet[page 47]{daley:vere-jones:2003}, a point process is said to be
$simple$ if the probability is zero that there exists a time $t$ at which more
than one event occurs.  We do not assume that the counting processes are
simple. Thus we allow for the possibility that several transactions may occur at
the same time. The transaction times $t_{i,k}$ are related to the point process
by the following duality.
\begin{align*}
  N_i(t) = k \Leftrightarrow t_{i,k} \leq t < t_{i,k+1} \; .
\end{align*}
The durations are then defined by
\begin{align*}
   \tau_{i,k} = t_{i,k}-t_{i,k-1} \; .
\end{align*}
If the process is simple, then $N_i(t_{i,k})=k$. Otherwise, it only holds that
$N_i(t_{i,k}) \geq k$. 

There is no requirement that either the point processes $N_i$ or the durations
$\{\tau_k\}$ be stationary under $\pr$. Therefore we make the following
ergodicity-type assumptions.

\begin{assumption}
  \label{hypo:lfgn-dates}
  The sequences $\{t_{i,k}\}$ are nondecreasing and there exists $\lambda_i \in
  (0,\infty)$ such that
\begin{align}
  \label{eq:ergodic-dates}
  t_{i,k}/k \plim 1/\lambda_i \; .
\end{align}
\end{assumption}

When the counting processes are simple, this is equivalent to $N_i(t)/t \plim
\lambda_i$. Since we do not assume simplicity, we must introduce an additional
assumption.

\begin{assumption}
  \label{hypo:lfgn-pp}
  $N_i(t)/t \plim \lambda_i$.
\end{assumption}

If the counting processes are defined from stationary ergodic durations, then
Assumption~\ref{hypo:lfgn-dates} holds. If the couting processes are moreover
simple, then Assumption~\ref{hypo:lfgn-pp} also holds.  Conversely, if $N_i$ is
stationary and ergodic then Assumption~\ref{hypo:lfgn-pp} holds, and if $N_i$ is
moreover simple, then Assumption~\ref{hypo:lfgn-dates} also holds.

It should be stressed that stationarity of durations and
of the point process cannot hold simultaneously under the same probability
measure (except for the Poisson point process). See for instance
\cite{baccelli:bremaud:2003} or \cite{daley:vere-jones:2003} for the
mathematical theory and \cite{deo:hurvich:soulier:wang:2009} for an econometric
interpretation.  
%
%
We now give an example which illustrates this duality.

\begin{example}
  \label{xmpl:LMSD}
  For the LMSD model, consider a probability measure (called the Palm measure)
  $P^0$ on $(\Omega,\mathcal F)$ under which 
  $\tau_{k} = \epsilon _{k} \mathrm e^{\sigma Y_{k}}$ where $\sigma$ is a
  positive constant, $\{\epsilon_{k}\}$ is an i.i.d.~sequence of almost surely
  positive random variables with finite mean and $\{Y_{k}\}$ is a stationary
  standard Gaussian process, independent of $\{\epsilon_{k}\}$, whose covariance
  function goes to 0 at infinity.  It follows from the latter assumption and
  Gaussianity that the process $\{Y_{k}\}$ is ergodic \cite{MR543837}, hence so
  is $\{\tau_{k}\}$.  Suppose now that the restriction of $\pr$ to the
  sigma-field generated by $N$ is equal to $P^0$, then
  Assumption~\ref{hypo:lfgn-dates} holds with $\lambda^{-1} = \mu \mathrm
  e^{\sigma^2/2}$ where $\mu$ is the expectation of $\epsilon_0$ under $P^0$,
  and since the durations are almost surely positive, so does
  Assumption~\ref{hypo:lfgn-pp}. Now, by the Palm duality theory, there exists a
  probability measure $P$ under which the associated point process $N$ is
  stationary (in which case, as mentioned above, the durations are no longer
  stationary). If the restriction of $\pr$ to the sigma-field generated by $N$
  is equal to $P$, then Assumptions~\ref{hypo:lfgn-dates} and~\ref{hypo:lfgn-pp}
  still hold with the same $\lambda$ as defined above.

\end{example}

\begin{example}
  \label{xmpl:ACD}
   The ACD model proposed in \cite{MR1639411} is
\begin{align}
   \label{eq:acd}
   \tau_k & = \psi_k\epsilon_k,\qquad
   \psi_k=\omega+\alpha\tau_{k-1}+\beta\psi_{k-1},\qquad k\in\mathbb{Z},
\end{align}
where $\omega>0$ and $\alpha,\beta\geq 0$, $\{\epsilon_k\}_{k=-\infty}^\infty$
is an i.i.d.~sequence of almost surely positive random variables with mean~1. If
$\alpha+\beta<1$ then the sequence $\{\tau_k\}_{k=-\infty}^\infty$ is strictly
stationary and ergodic and has finite mean $\omega/(1-\alpha-\beta)$.  As in the
above example, we can alternatively assume that the durations form a stationary
ACD process under $\pr$ or that the associated point process is stationary under
$\pr$.  In both cases, Assumptions~\ref{hypo:lfgn-dates} and~\ref{hypo:lfgn-pp}
hold with $\lambda = (1-\alpha-\beta)/\omega$.
\end{example}

We now explain
how time deformations can be used to obtain a nonstationary, possibly non simple
point process from a stationary ergodic simple one, while
Assumptions~\ref{hypo:lfgn-dates} and~\ref{hypo:lfgn-pp} are preserved. Let
$\tilde N(\cdot)$ be a simple, stationary and ergodic counting process on
$\mathbb R$ with intensity $\tilde\lambda\in(0,\infty)$ and let $f$ be a
deterministic or random function such that $f$ is nondecreasing and has
c\`adl\`ag paths with probability one.  Define then
\[
N(t)=\tilde N(f(t)) \; .
\]
If the function $f$ is random, we assume moreover that it is independent of the
counting process $\tilde N$.

The function $f$ is used to speed up or slow down the trading clock. To
incorporate dynamic intraday seasonality in volatility, the same time
deformation can be used in each trading period (of length, say, $T$), assuming
that $f(t)$ has a periodic derivative (with period $T$ and with probability
one), for example, $f(t)=t+.5\sin(2\pi t/T)$. Fixed nontrading intervals, say,
$t \in [T_1,T_2)$, could be accommodated by taking $f(t)=f(T_1)$ for $t \in
[T_1,T_2)$ so that $f(t)$ remains constant for $t$ in this interval, and then
taking $f(T_2) > f(T_1)$ so that $f(t)$ jumps upward when trading resumes at
time $T_2$. The jump allows for the possibility of one or more transactions at
time $T_2$, potentially reflecting information from other markets or assets that
did trade in the period $[T_1,T_2)$.

If the values of some series are only recorded at specific time points (e.g.,
quarterly in the case of certain macroeconomic series) this could be handled by
taking the corresponding $f(t)$ to be a pure-jump function.  This provides scope
for considering two (or more) series, some of which are observed continuously,
others at specific times, though not necessarily contemporaneously. In future
work, we hope to explore this scenario in detail, and its possible connections
with the MIDAS methodology, see \cite{MR2288650}.

The use of the time-varying intensity function $f$ renders the counting process
$N$ nonstationary.  Since it is possible that $f$ has (upward) jumps, the $N$
may also not be simple even though the $\tilde N$ are simple.  We now show,
however, that if $\tilde N$ satisfies Assumptions~\ref{hypo:lfgn-dates}
and~\ref{hypo:lfgn-pp}, then so does the time deformed $N$ under some restrictions on $f$.

\begin{lemma}
  \label{lem:TimeDeformation-nonstationaire}
  Assume that $f$ is a nondecreasing (random) function such that $t^{-1}f(t)
  \plim \gamma \in(0,\infty)$ and
\[
\sup_{t\geq 0} |f(t) - f(t^-)| \leq C
\]
with probability one, where $C \in (0,\infty)$ is a deterministic constant. Let
$\tilde N$ be a point process such that Assumptions~\ref{hypo:lfgn-dates}
and~\ref{hypo:lfgn-pp} hold for some $\lambda\in(0,\infty)$. Let $N$ be the counting
process defined by $N(\cdot)=\tilde N(f(\cdot))$.  Then
Assumptions~\ref{hypo:lfgn-dates} and~\ref{hypo:lfgn-pp} hold for $N$ with
$\lambda=\tilde\lambda\gamma$.
\end{lemma}

Stationarity of $\tilde N$ is not required in this Lemma. For example, one could
equally well do the time deformation on a counting process that corresponds to a
stationary duration sequence (and hence the counting process is not stationary).

In order to show that our results on estimation of the cointegrating parameter
(under weak fractional and standard cointegration) hold in this time deformation
framework, we will in Lemma~\ref{lem:time-deformation} make further assumptions
on the function $f$. These assumptions mathematically embody natural economic
constraints, viz.\ minimum duration of trading and nontrading periods, maximum
duration of nontrading periods and non stoppage of trading time during trading
periods.

We now state our assumptions on the return shocks.

\begin{assumption}
  \label{hypo:efficient-shocks}
  The efficient shocks $\{e_{i,k}\}$ are mutually independent i.i.d. sequences
  with zero mean and variance $\sigma_{i,e}^2$.
\end{assumption}

Although many of our results would continue to hold if the i.i.d.~part of
Assumption~\ref{hypo:efficient-shocks} were replaced by a weak-dependence
assumption, we maintain the i.i.d.~assumption here in keeping with the economic
motivation for the model as provided by \cite{hurvich:wang:2009} that in the
absence of the microstructure shocks and in the absence of any dependence of the
efficient shocks on the counting processes, each of the log price series would
be a martingale with respect to its own past. Since the trades of Asset 1 are
not synchronized in calendar time or in transaction time with those of Asset 2,
it seems reasonable to assume that the two efficient shock series are mutually
independent, as we have done in Assumption~\ref{hypo:efficient-shocks}.

The following assumption implies that the microstructure noise does not affect
the limiting distribution of the log prices. 

\begin{assumption}[Microstructure Noise]
  \label{hypo:microstructure}
  The microstructure noise sequences $\{\eta_{i,k}\}$ satisfy $n^{-1/2}
  \sum_{k=1}^{[n\cdot]} \eta_{i,k} \func 0$.
\end{assumption}

This assumption simply enforces the negligibility of the microstructure noise
with respect to the long term behaviour of the log price and indeed allows us to
estabish in Theorem~\ref{theo:fclt-levels} that the log price behaves
asymptotically as a random walk (and returns at long horizons behave like a
martingale difference sequence), consistent with what is assumed in most
econometric literature.

Dependence between the counting processes and return shocks allows
for leverage effects (for example, a correlation between a return in
one time period and a squared return in a subsequent time period). A
transaction-level model yielding a leverage effect was proposed (but
justified only with simulations) in \cite{hurvich:wang:2009preprint}.
Models where the point process need not be independent of the return
shocks were discussed in \cite{MR1821828} in the context of
option pricing with marked point processes.

We do not make any assumption of independence between the counting
processes and the microsctucture shocks, unless explicitly noted
otherwise.  We will, however, assume that the counting processes are
independent of the efficient shocks except in
Theorem~\ref{theo:fclt-levels} and in Section~\ref{sec:spurious}.

Assumption \ref{hypo:microstructure} is all we need to assume about the
microstructure noise in order to obtain a limit law for the log price series
(such as Theorem \ref{theo:fclt-levels} below). However, in order to discuss
properties of estimators of the cointegrating parameter it is necessary to make
more specific assumptions on the degree of cointegration. In
\cite{hurvich:wang:2009preprint,hurvich:wang:2009}, three different cases were
considered, according to the strength of the memory of the microstructure noise
sequences. These cases were labeled as weak fractional cointegration, strong
fractional cointegration and standard cointegration. In the current context,
where there may be a dependence between return shocks and counting processes,
special care is needed in defining the strong fractional and standard
cointegration cases, as long memory in durations may affect the rate of
convergence of estimators of the cointegrating parameter in these cases. On the
other hand, we will define weak fractional cointegration essentially as in
\cite{hurvich:wang:2009}.

\begin{assumption}
   \label{as:independence-shocks}
   The shocks $\{e_{1,k}\}_{k=-\infty}^{\infty}$,
   $\{e_{2,k}\}_{k=-\infty}^{\infty}$, $\{\eta_{1,k}\}_{k=-\infty}^{\infty}$ and
   $\{\eta_{2,k}\}_{k=-\infty}^{\infty}$ are mutually independent.
\end{assumption}

Mutual independence of the efficient and microstructure shock series of a given
asset can be justified on economic grounds, and is often made in the econometric
literature for calendar-time models. See, e.g., \cite{MR2468558}. Mutual
independence of the two microstructure series is justified by the lack of
synchronization of the trading times of the two assets.

We now discuss the weak fractional cointegration case. For
$H\in(0,1)$, let $B_H$ denote the standard fractional Brownian motion (FBM) with
Hurst index $H$, i.e. the zero mean Gaussian process with almost surely
continuous sample paths and covariance function
\begin{align*}
  \mathrm {cov} (B_H(s),B_H(t)) = \frac12 \left\{s^{2H} - |t-s|^{2H} + t^{2H} \right\} \; .
\end{align*}
For $H=1/2$, $B_{1/2}$ is the standard Brownian motion.

\begin{assumption}[Weak Fractional Cointegration]
  \label{hypo:microstructure-weak}
There exists $H \in (0,1/2)$ such that
\begin{align*}
  n^{-H} \sum_{k=1}^{[n\cdot]} \eta_{i,k} \func c_i B_H^{(i)}
\end{align*}
where $c_1,c_2$ are nonnegative constants, not both zero and $B_H^{(1)}$ and
$B_H^{(2)}$ are independent standard fractional Brownian motions with common
Hurst index~$H$.
\end{assumption}

Assumption~\ref{hypo:microstructure-weak} is a strengthening of
Assumption~\ref{hypo:microstructure} and is needed to establish the asymptotic
distribution of estimators of the cointegrating parameter $\theta$. Whereas
Assumption~\ref{hypo:microstructure} allows for cointegration, an assumption
such as Assumption~\ref{hypo:microstructure-weak} is a necessary element for
defining the degree of cointegration (the rate at which deviations from the long
run cointegrating relationship disappear), and is consistent with the finding
reported in Section~\ref{sec:intro} that the cointegrating residuals have weaker
memory than a random walk.

Under Assumption~\ref{as:independence-shocks}, the independence of all the noise
series implies that all the previous convergences hold jointly.  The situation
where one of the constants $c_1$ or $c_2$ is zero could arise naturally if the
memory in one of the microstructure noise series is weaker than for the other.

In the case of weak fractional cointegration, we can define the memory parameter
of the microstructure noise series as $d_{\eta}=H-1/2 \in (-1/2,0)$, and the
degree of fractional cointegration (i.e. the rate of convergence of partial sums
of the cointegrating error) is completely determined by $d_{\eta}$. More
precisely, in this case the difference between the memory parameters of the
series of log prices and the cointegrating error (observed, say, at
equally-spaced intervals of calendar time) $y_1(j)-\theta y_2(j)$, is
$-d_{\eta}$. This holds regardless of any dependence between the counting
processes and the microstructure shocks.

Next we discuss  strong fractional and standard cointegration. We start with
the assumption that, for $i=1,2$, $\eta_{i,k}=\xi_{i,k}-\xi_{i,k-1}$
where $\{\xi_{i,k}\}$ satisfy $\sup_{k}\esp[\xi_{i,k}^2] < \infty$.
It then follows that the cointegrating error at time $j$ is

\begin{multline}
  y_1(j) - \theta y_2(j) \\
  = \sum_{k=N_1(t_{2,N_2(j )})+1}^{N_1(j )} e_{1,k} - \theta
  \sum_{k=N_2(t_{1,N_1(j )})+1}^{N_2(j )} e_{2,k} + \xi_{1,N_1(j )}-\xi_{1,0} -
  \theta (\xi_{2,N_2(j )}-\xi_{2,0}). \label{DiffCointErrors}
\end{multline}
Under the assumptions we will make in this paper, and also under the assumptions
made in \cite{hurvich:wang:2009,hurvich:wang:2009preprint}, the first two terms
on the righthand side of (\ref{DiffCointErrors}) are weakly dependent, so the
degree of cointegration is determined by the rate of convergence of partial sums
of $\xi_{i,N_i(j )}$.

Thus we will need to study the sequence $\xi_{i,N_i(j)}$.  We do not assume that
the microstructure shocks are independent of the counting processes. Thus, even
if the microstructure shocks have zero mean, it may hold that
$\esp[\xi_{i,N_i(j)}] \ne 0$.

In view of the discussion above it is clear that in order to specify the degree
of cointegration in the strong fractional and standard cointegration cases, it
is necessary to make an assumption on calendar-time aggregates of
$\xi_{i,N_i(j)}$. This is in contrast to Assumption 2.6 above where the degree
of fractional cointegration is specified in terms of properties of
transaction-level aggregates of the microstructure noise.

\begin{assumption} [Strong fractional and standard cointegration]
  \label{hypo:microstructure-strong}
  The microstucture noise sequences $\{\eta_{i,k}\}$ can be expressed as
  $\eta_{i,k} = \xi_{i,k}-\xi_{i,k-1}$. There exist $H\in[1/2,1)$, constants
  $\mu_1^*$, $\mu_2^*$ and nonnegative constants $c_1,c_2$, not both zero, such
  that
  \begin{align*}
    n^{-H} \sum_{j=1}^{[n\cdot]} \{\xi_{i,N_i(j)} - \mu_i^*\} \func c_i
    B_H^{(i)}
  \end{align*}
  where $B_H^{(1)}$ and $B_H^{(2)}$ are independent fractional Brownian motions
  with Hurst index $H$.

\end{assumption}
The case $H>1/2$ corresponds to strong fractional cointegration whereas the case
$H=1/2$ corresponds to standard cointegration.

It might be hard to verify Assumption~\ref{hypo:microstructure-strong} unless
the durations are integer valued. Since commonly-used duration models do not
have integer-valued durations, we will introduce a modification of the
estimators which involves integrals instead of sums, thus allowing us to avoid
this restriction. This change requires a corresponding modification of
Assumption~\ref{hypo:microstructure-strong}.

\begin{assumption} [Strong fractional and standard cointegration]
  \label{hypo:microstructure-strong-taper-continuous}
  The microstucture noise sequences $\{\eta_{i,k}\}$ can be expressed as $\eta_{i,k} =
  \xi_{i,k}-\xi_{i,k-1}$. There exist $H\in[1/2,1)$, constants $\mu_1^*$,
$\mu_2^*$ and nonnegative constants $c_1,c_2$, not both zero, such that
  \begin{align*}
    n^{-H} \int_{0}^{n\cdot} \{\xi_{i,N_i(s)} - \mu_i^*\} \, \mathrm d s
  \func c_i B_H^{(i)} \; .
  \end{align*}

\end{assumption}

In their strong fractional cointegration case, \cite{hurvich:wang:2009} assumed,
for $d_{\eta} \in (-1,-1/2)$, that $\mathrm{cov}(\xi_{i,k},\xi_{i,k+j}) \sim K
j^{2d_{\xi}-1}$ as $j \to \infty$ where $K>0$ and $d_{\xi} = d_{\eta}+1
\in (0,1/2)$. They then showed (in their Lemma 3), under the assumptions made
there, that $\mathrm{cov}(\xi_{i,N_i(k)},\xi_{i,N_i(k+j)}) \sim K^\prime j^{2d_{\xi}-1}$
as $j \to \infty$ where $K^\prime >0$, so that the degree of fractional
cointegration was completely determined by the rate of decay of
$\mathrm{cov}(\xi_{i,k},\xi_{i,k+j})$. However, the proof of this result relied on the
assumption that the microstructure shocks are independent of the counting
processes, an assumption which we do not make here.

We next provide an example showing that under dependence between the
microstructure shocks and the counting processes, it is possible for
$\{\xi_{i,k}\}$ to be weakly dependent, and yet the rate of
convergence of suitably normalzed integrals of the process
$(\xi_{i,N_i(t)} : t \geq 0)$ is determined by the degree of long
memory in durations. Suppressing the $i$ subscript, we have the
following lemma.

\begin{lemma}
  \label{lem:propogation}
  Suppose that $\{\tau_k\}$ is given by the LMSD model $\tau_k=\epsilon_k\mathrm
  e^{Y_k}$, $\{\epsilon_k\}$ are i.i.d.  standard exponential, independent of
  the stationary standard Gaussian series $\{Y_k\}$, which satisfies $\mathrm
  {cov} (Y_0,Y_n) \sim cn^{2H_\tau-2}$ where $c>0$, and $H_\tau \in
  (1/2,3/4)$. Define $\xi_k = Y_{k+1}^2-1$.  Then the autocovariance function of
  $\{\xi_k\}$ is summable and there exists $c'>0$ such that
\begin{align}
  \label{eq:weak-dep}
  n^{-1/2} \sum_{k=1}^{[n\cdot]} \xi_k \func c' B \; .
\end{align}
Nevertheless, the randomly-indexed continuous-time process $\xi_{N(t)}$ has long
memory in the sense that there exists a constant $\mu^*$, as well as a constant
$c^{\prime\prime}$, such that
\begin{align}
  \label{eq:lrd-random-index}
  n^{-H_\tau}  \int_0^{n\cdot} \{\xi_{N(s)} - \mu^*\} \, \mathrm d s \func
  c^{\prime\prime} B_{H_\tau} \; .
\end{align}
\end{lemma}

Lemma \ref{lem:propogation} shows that long memory in durations can induce the
same degree of long memory in the cointegrating error (\ref{DiffCointErrors}) in
calendar-time, even though the microstructure shocks, which are the source of
the cointegration, have short memory as a sequence in transaction time. In Lemma
\ref{lem:propogation}, this phenomenon was achieved by imposing a particular
functional relationship between the (zero mean) microstructure noise and the
persistent component of the durations, $\xi_k=Y_{k+1}^2-1$. This relationship
implies a leverage effect, since $\mathrm{corr}(\xi_k, \tau_{k+1}) =
1/\sqrt{2(e-1)} \approx .539 > 0$.
In other words, a strongly negative microstructure shock to the return leads to
a tendency of the next observed duration, as well as subsequent durations, to be
shorter than average. Such a string of short durations increases the volatility,
e.g., the expectation of squared calendar-time returns, as shown, for example,
under a particular return model in
\cite{deo:hurvich:soulier:wang:2009}. Furthermore, evidence that stock market
intertrade durations have long memory was provided in \cite{Deo20103715}.

In the absence of dependence between the counting processes and microstructure
noise series, in both cases of strong fractional and standard cointegration, the
memory of durations cannot affect the memory of the cointegrating error. See
Lemma~\ref{lem:strong-noleverage-continuous} for strong fractional cointegration
and Lemma~\ref{lem:standard-dependence-noleverage} for standard cointegration.

\section{Main results}

\label{sec:main}

\subsection{The long-run behavior of the bivariate log-price process}
\label{sec:LongRun}

With the assumptions made in Section~\ref{sec:model-assumptions}, the long-run
behavior of the bivariate process $y=(y_1,y_2)$ can be determined.
The following theorem shows that the log-prices are approximately
integrated. Even though independence is assumed between the various
shock series, the log-price process $y=(y(t)\colon t\geq 0)$ exhibits
a nontrivial variance-covariance structure which is determined by a
complex interplay of the model parameters.

\begin{theorem}
  \label{theo:fclt-levels}
  Under Assumptions~\ref{hypo:lfgn-dates}, \ref{hypo:lfgn-pp},
  \ref{hypo:efficient-shocks}, \ref{hypo:microstructure}, $n^{-1/2}
  (y_1(n\cdot),y_2(n\cdot)) \func \mathbb B$, where
  \begin{align}
   \label{eq:fclt-levels}
   \mathbb B = \left( \sigma_{1,e} \sqrt {\lambda_1} B_1 + \theta \sigma_{2,e} \sqrt
     {\lambda_2} B_2 \, , \; \theta^{-1} \sigma_{1,e}\sqrt
     {\lambda_1} B_1 + \sigma_{2,e} \sqrt {\lambda_2} B_2 \right) \; .
  \end{align}
  and $B_1$ and $B_2$ are independent standard Brownian motions.
\end{theorem}

In Theorem~\ref{theo:fclt-levels}, we have not assumed that the counting
processes are independent of either the efficient shocks or the microstructure
shocks.

 \cite{hurvich:wang:2009,hurvich:wang:2009preprint} have in their
Theorem 1 computed the long-run variances of $y_1(t)$ and $y_2(t)$ which are
given as $(\sigma_{1,e}^2 \lambda_1 + \theta^2 \sigma_{2,e}^2 \lambda_2 )t$ and
$(\sigma_{1,e}^2 \lambda_1 / \theta^2 + \sigma_{2,e}^2 \lambda_2 )t$,
respectively. Our theorem yields the variances as well as the covariances in the
limiting distribution of $(t^{-1/2}y(t)\colon t\geq 0)$. More importantly, our
theorem provides the limiting distribution itself for the (normalized) log-price
process $y$ which, in turn, can be used for asymptotic statistical
inference. Indeed, most of the subsequent results in this paper use Theorem
\ref{theo:fclt-levels} and its proof as a building block. In particular, a
slightly generalized version of this theorem directly yields asymptotics for
estimators in spurious regressions and therefore can be used to motivate tests
for the null hypothesis of no cointegration, as we describe in Section
\ref{sec:spurious}.

\subsection{OLS estimator of the cointegrating parameter}
\label{sec:ols-cointegration}

In this section, we derive the asymptotic behavior of the ordinary
least-squares estimator (OLS) of the cointegrating parameter
$\theta$. To do so, we assume that the log-price series are observed
at integer multiples of $\Delta t$. We will work here, without loss
of generality, with $\Delta t=1$ in order to keep the notation
simple. Then (\ref{eq:P1}) and~(\ref{eq:P2}) become
\begin{align*}
  y_1(j) & =\sum_{k=1}^{N_1(j)}(e_{1,k}+\eta_{1,k})+\theta
  \sum_{k=1}^{N_2(t_{1,N_1(j)})} e_{2,k}, 
  \\
  y_2(j)&=\sum_{k=1}^{N_2(j)}(e_{2,k}+\eta_{2,k})+\theta^{-1}
  \sum_{k=1}^{N_1(t_{2,N_2(j)})} e_{1,k}. 
\end{align*}
Regressing $y_1(1),\ldots,y_1(n)$ on $y_2(1),\ldots,y_2(n)$ without
intercept, we obtain the OLS estimator of $\theta$ as
\begin{equation} \label{ols}
\hat\theta_{n}^{\scriptscriptstyle{OLS}}=\frac{\sum_{j=1}^ny_2(j)y_1(j)}{\sum_{j=1}^ny_2^2(j)}.
\end{equation}
\cite{hurvich:wang:2009,hurvich:wang:2009preprint} have shown in their Theorem 6
(under conditions that are for the most part stronger than the ones
we assume here) that $\hat\theta_n^{\scriptscriptstyle{OLS}}$ is weakly consistent for $\theta$
and obtained a lower bound on the rate of convergence in the case of
weak fractional, strong fractional and standard cointegration. The
exact limit distributions, however, were not given. We fill in this
gap next for weak fractional cointegration.

\subsubsection{OLS and inference for the cointegrating parameter under weak fractional cointegration}

\begin{theorem}
   \label{theo:ols-weak}
   Let Assumptions \ref{hypo:lfgn-dates}, \ref{hypo:lfgn-pp},
   \ref{hypo:efficient-shocks}, \ref{as:independence-shocks} and
   \ref{hypo:microstructure-weak} hold. Assume in addition that the counting
   processes $N_1$ and $N_2$ are mutually independent and independent of the
   efficient shocks and there exists a constant $C$ such that
  \begin{align}
    \label{eq:pp-renforce}
    \sup_{s\geq0} \esp[ t_{i,N_i(s)+1}-s] \leq C \; .
  \end{align}
  Then
  \begin{align*}
    n^{1/2-H} (\hat \theta_{n}^{\scriptscriptstyle{OLS}}-\theta) \dlim \Sigma_1 \frac{\int_0^1 B(t) B_H(t)
      \, \mathrm d t}{\int_0^1 B^2(t) \, \mathrm d t}
  \end{align*}
  where $B$ is a standard Brownian motion, $B_H$ is a fractional Brownian motion,
  independent of $B$, and
\begin{align*}
  \Sigma_1^2 = \frac{c_1^2\lambda_1^{2H} + \theta^2 c_2^2 \lambda_2^{2H}} {\theta^{-2}
    \lambda_{1} \sigma_{1,e}^2 + \lambda_2 \sigma_{2,e}^2} \; .
\end{align*}

\end{theorem}

The result in Theorem \ref{theo:ols-weak} is similar to that obtained in
\citet[Proposition 6.5, formula (6.8)]{MR1869235}, under their Assumption 6.1,
for which a sufficient condition (their formula (6.5)) was verified in
\cite{MR1741198} to hold for weak (but not strong) fractional cointegration in
the case where the process is linear with respect to i.i.d.~innovations.

Next, we provide sufficient conditions for Condition~(\ref{eq:pp-renforce}) to
hold for the LMSD and ACD models under both frameworks for stationarity
considered in Examples~\ref{xmpl:LMSD} and~\ref{xmpl:ACD}. 
  \begin{itemize}
  \item Note first that as long as the point processes $N_i$ are stationary
    under $\pr$, then condition~(\ref{eq:pp-renforce}) holds provided that
    $\esp[t_1]<\infty$ (equivalently $E^0[\tau_0^2]<\infty$ where $E^0$
    represents expectation with respect to the Palm measure), since the forward
    recurrence time $\{t_{i,N_i(s)}-s\}$ is then stationary.
  \item If the durations form a stationary LMSD sequence under $\pr$, then
    Lemma~\ref{lem:check-lmsd} shows that if $\esp[\epsilon_0^q]<\infty$ for all
    $q\geq1$, then (\ref{eq:pp-renforce}) holds.  
  \item If the durations form a stationary ACD sequence under $\pr$, then
    Lemma~\ref{lem:moment-frt} shows that~(\ref{eq:pp-renforce}) holds as long
    as $\esp[\tau_0^3]<\infty$. By \citet[Corollary~6]{MR1885348}, this holds
    true if $\esp[(\beta+\alpha \epsilon_t)^3]<1$.
  \item If $\tilde N$ is a point process that satisfies
    Condition~(\ref{eq:pp-renforce}) under $\pr$, and if $f$ is a time deformation
    function that satisfies the economically justified assumptions of
    Lemma~\ref{lem:time-deformation}, then the time deformed point process $N$
    (defined by $N(t)=\tilde N(f(t))$ still satisfies~(\ref{eq:pp-renforce}).
\end{itemize}

Next, we consider inference for the cointegrating parameter $\theta$
under weak fractional cointegration. It is seen from Theorem
\ref{theo:ols-weak} that the asymptotic distribution of $n^{1/2-H}
(\hat \theta_n^{OLS}-\theta)$ depends on a variety of unknown
quantities. To alleviate the dependence on nuisance parameters and to
thereby facilitate inference on $\theta$, we consider the
$t$-statistic for testing the null hypothesis $H_0 : \theta =
\theta_0$. If the null hypothesis is true, then the $t$-statistic is
given by
\[
t_n=\frac{\hat \theta_n^{OLS}-\theta}{\hat \sigma_{\hat \theta_n^{OLS}}}
\]
where $\hat \sigma_{\hat \theta_n^{OLS}}$ is the traditional
estimated standard error for $\hat \theta_n^{OLS}$, with
\[
\hat \sigma_{\hat \theta_n^{OLS}}^2= \frac{n^{-1} \sum_{j=1}^n [y_1(j)-\hat \theta_n^{OLS}y_2(j)]^2}{\sum_{j=1}^n y_2^2(j)}\,\,\,.
\]
We have the following theorem.
\begin{theorem} \label{theo:t}
Under the assumptions of Theorem \ref{theo:ols-weak},
\[
n^{-1/2} t_n \dlim \left [ \frac{\int_0^1 B^2(t)dt \int_0^1 B_H^2(t) dt}{\left [ \int_0^1 B(t)B_H(t) dt \right ]^2} - 1\right ]^{-1/2},
\]
where $B$ is a standard Brownian motion and $B_H$ is a standard
fractional Browninan motion, independent of $B$.
\end{theorem}
Note that by the Cauchy-Schwarz Inequality,
\[
\left | \int_0^1 B(t)B_H(t) dt \right | ^2 \leq \int_0^1 B^2(t)dt \int_0^1 B_H^2(t) dt \,\,\, ,
\]
and in fact the inequality above is strict since $B$ and $B_H$ are
mutually independent, so that the limiting distribution in Theorem
\ref{theo:t} is well-defined.

Note that the limiting distribution in Theorem \ref{theo:t} depends
only on $H$. Thus, if $H$ can be consistently estimated, one can
conduct asymptotically valid hypothesis tests for $\theta$ based on
the test statistic $n^{-1/2} t_n$. The asymptotic null distribution
of the test statistic is given by Theorem \ref{theo:t}, and it is
easily seen that if $\theta \neq \theta_0$ then $n^{-1/2} t_n$
diverges.

We now show that $H$ can indeed be consistently estimated, using an
aggregation method considered in \cite{MR1670119}. Let $m$ be an
integer sequence such that $1/m+m/n \to 0$. We work with the
differences of the cointegrating residuals, $x_j = \Delta (y_1(j) -
\hat \theta y_2(j))$, where $\Delta$ is the differencing operator.
Divide the data set into contiguous, non-overlapping blocks of size
$m$. The $k$th block average is
\[
X_k^{(m)} = \frac{1}{m} \sum_{t=1}^m x_{t+(k-1)m} \,\,\, .
\]
Denote the sample variance of these block averages by
\[
s_m^2 = \lfloor n/m \rfloor ^{-1} \sum_{k=1}^{\lfloor n/m \rfloor} \left ( X_k^{(m)} \right )^2 \,\,\,.
\]
Now, define $\hat H = 1 + \frac{ \log s_m^2} {2\log m}$. The next theorem states
that $\hat H$ is a consistent estimator of $H$ under assumptions that imply
those of Theorem~\ref{theo:ols-weak}, but are more restrictive. 
We make these restrictions to facilitate a reasonably simple proof. 

\begin{theorem} 
  \label{theo:Aggregation} 
  Let Assumptions~\ref{hypo:efficient-shocks} and~\ref{as:independence-shocks}
  hold.  Assume that $\eta_{i,t} = \sigma_i\{B_{i,H}(t)-B_{i,H}(t-1)\}$, $i=1,2$
  where $B_{i,H}$ are mutually independent standard fractional Brownian
  motions with $H\in(0,1/2)$. Assume that $N_1$ and $N_2$ are stationary
  mutually independent ergodic point processes, independent of the processes
  $B_{i,H}$,  such that $\esp[N_i^4(1)]<\infty$ and
  \begin{align}
  \label{eq:moment-cond-8}
  \sup_{t\geq 2} \esp\left[ \left(
      \frac{N_i(t)}t\right)^{4H-8}\mathbf1_{\{N_i(t)>0\}} \right] < \infty \; .
  \end{align}
  Then, $\hat H \plim H$.
\end{theorem}

We now continue the discussion of the LMSD and ACD models (see
Example~\ref{xmpl:LMSD} and~\ref{xmpl:ACD} and the discussion immediately after
Theorem~\ref{theo:ols-weak}). The proof of the claims is in
Section~\ref{sec:additional}, following Lemma~\ref{lem:negative-moments} which
it uses.

\begin{itemize}
\item Assume that the point process $N$ is stationary under the probability
  $\pr$ and that under the Palm probability $P^0$ the durations form an LMSD
  sequence $\{\tau_k\}$ with memory parameter $H_\tau\in(1/2,1)$ as in
  Example~\ref{xmpl:LMSD}.  If $E^0[\epsilon_0^{9-4H}]<\infty$
  condition~(\ref{eq:moment-cond-8}) holds. If $E^0[\epsilon_0^q]<\infty$ with
  $q>4/(1-H_\tau)$, then $\esp[N^4(1)]<\infty$.
\item Assume that point process $N$ is stationary under the probability $\pr$
  and that under the Palm probability $P^0$ the durations form an ACD sequence
  $\{\tau_k\}$ as in Example~\ref{xmpl:ACD}.  If moreover $\epsilon_0$ admits a
  strictly positive density on $[0,\infty)$ and $E^0[\tau_1^{9-4H+\eta}]<\infty$
  for some $\eta>0$, then $\esp[N^4(1)]<\infty$ and condition~(\ref{eq:moment-cond-8})
  holds. A condition for $E^0[\tau_1^{9-4H+\eta}]<\infty$ is $E^0[(\alpha
  \epsilon_0+\beta)^{9-4H+\eta}] < 1$. 
\end{itemize}

\subsubsection{OLS under strong fractional and standard cointegration}

We now consider the case where the microstructure noise series
$\{\eta_{i,k}\}$ are differences of strongly or weakly dependent
processes $\{\xi_{i,k}\}$.

\begin{theorem}
  \label{theo:ols-strong}
  Let Assumptions~\ref{hypo:lfgn-dates}, \ref{hypo:lfgn-pp}, \ref{hypo:efficient-shocks},
  \ref{as:independence-shocks}, and \ref{hypo:microstructure-strong} hold.
  Assume moreover that
  \begin{itemize}
  \item the efficient shocks are i.i.d. Gaussian,
  \item the counting processes $N_1$ and $N_2$ are independent of each other and
    independent of the microstructure noise sequences and of the efficient
    shocks and there exists a constant $C$ such that
  \begin{align}
    \label{eq:pp-renforce-2}
    \sup_{t\geq0} \esp[ (t_{i,N_i(t)+1}-t)^2 ] \leq C \; .
  \end{align}
  \item $\esp[\xi_{i,k}]=0$, $\sup_{k}\esp[\xi_{i,k}^2] < \infty$, and
    $\xi_{i,0}=0$.
  \end{itemize}
Then,
\begin{itemize}
\item if $1/2<H<1$,
\begin{gather}
  n^{3/2-H} (\hat \theta_{n}^{\scriptscriptstyle{OLS}} -\theta) \dlim
  \sqrt{\frac{c_1^2+\theta^2c_2^2}{\theta^{-2}\lambda_{1}\sigma_{1,e}^2+\lambda_2\sigma_{2,e}^2}}
  \, \frac{\int_0^1 B(s) \, \mathrm d B_H(s) } { \int_0^1 B^2(s) \, \mathrm d s}
  \; , \label{eq:limit-strong}
\end{gather}
where $B$ is standard Brownian motion independent of the standard fractional
Brownian motion $B_H$;
\item if $H=1/2$, $ n (\hat \theta_{n}^{\scriptscriptstyle{OLS}} -\theta) = O_P(1)$.
\end{itemize}
\end{theorem}

The rate of convergence obtained in the standard cointegration case improves on
the one obtained by \cite{hurvich:wang:2009}.  The sufficient conditions
for~(\ref{eq:pp-renforce}) discussed after Theorem~\ref{theo:ols-weak} become
sufficient conditions for~(\ref{eq:pp-renforce-2}) after augmenting by 1 the
exponent in the moment conditions appearing there.  The assumptions in Theorem
\ref{theo:ols-strong} are quite strong, ruling out leverage effects and
providing one motivation for our subsequent consideration of tapered estimators.


\subsection{A Tapered Estimator of the Cointegrating parameter}
\label{sec:Taper}

Even in existing discrete-time models for cointegration the OLS estimator lacks
any particular optimality properties. Here we consider an estimator based on
discrete Fourier transforms of the tapered differences of $y_1(j)$, $y_2(j)$,
$1 \leq j \leq n$.  It was shown in \cite{MR2002300} that this estimator can
have a faster rate of convergence than OLS in certain cases of fractional
cointegration. In the weak fractional cointegration case, our limit results for
the tapered estimator (Theorem~\ref{theo:taper-weak}) are obtained under
identical conditions as those assumed in Theorem~\ref{theo:ols-weak} for
OLS. However, under strong fractional and standard cointegration, the conditions
for our results on the tapered estimator (Theorem~\ref{theo:taper-strong}) allow
for leverage, unlike the corresponding theorem for OLS.

We introduce all relevant notation using a generic time series
$\{x_j\}_{j=-\infty}^\infty$. Let $h \colon I\mapsto\mathbb{R}$ be a general
continuous taper function on an open interval $I$ containing $[0,1]$ such that
$h(0)=h(1)=0$.  For $\ell=1,2,\dots$, denote by $\omega_\ell = 2\pi \ell/n$ the
Fourier frequencies.  The tapered DFT of $\{ x_j\}_{j=-\infty}^\infty$ with
taper function $h$ is defined by
\begin{align*}
  d_{x,\ell} = \sum_{j=1}^n h\Big(\frac jn\Big) x_j \, \mathrm e^{\mathrm i j
    \omega_\ell} = \sum_{j=1}^n h_\ell\Big(\frac jn\Big) x_j .
\end{align*}
where $h_\ell(t) = h(t) \mathrm e^{2\pi\mathrm i \ell t}$.  Denote by $\{\Delta
x_j\}_{j=-\infty}^\infty$ the first difference of the series $\{x_j\}$, where
$\Delta x_j=x_j-x_{j-1}$. We define the tapered DFT of the first difference
$\{\Delta x_j\}_{j=-\infty}^\infty$ with taper function $h$ by
\begin{align}
 \label{eq:taperedDFT}
  d_{\Delta x,\ell} = \sum_{j=1}^n h\Big(\frac jn\Big) \Delta x_j \, \mathrm
  e^{\mathrm i j \omega_\ell} = \sum_{j=1}^n h_\ell\Big(\frac jn\Big)\Delta x_j
  .
\end{align}
In our setting, we observe the cointegrated component processes $y_1$ and $y_2$
at equidistant sample points. Defining the cointegrating error
$z_j=y_{1,j}-\theta y_{2,j}$ and following \cite{MR2011676}, we can now
introduce the estimator
\begin{align*}
  \hat\theta_n^{Tap} = \mathrm{Re}(\tilde\theta_n) \; ,
\end{align*}
where $\mathrm{Re}(z)$ signifies the real part of a complex number $z=a+\mathrm
i b$ and, letting $\bar z=a-\mathrm ib$ be the complex conjugate of $z$,
\[
\tilde \theta_n = \frac{\sum_{\ell=1}^m d_{\Delta y_1,\ell} \, \bar d_{\Delta
    y_2,\ell}}{\sum_{\ell=1}^m |d_{\Delta y_2,\ell}|^2} \; .
\]
Therein, any tapered DFT of differenced sequences is defined according to
\eqref{eq:taperedDFT}. Note that $\hat\theta_n^{Tap}$ is the real part of the ratio of
the averaged tapered cross-periodogram between the series $y_1$ and $y_2$ and
the averaged tapered periodogram of the series $y_2$.
\subsubsection{Discrete tapered estimator under weak fractional cointegration}

\begin{theorem}
  \label{theo:taper-weak}
  Let Assumptions \ref{hypo:lfgn-dates}, \ref{hypo:lfgn-pp}, \ref{hypo:efficient-shocks},
  \ref{as:independence-shocks} and \ref{hypo:microstructure-weak} hold. Assume
  moreover that the counting processes are mutually independent and independent
  of the efficient shocks and~(\ref{eq:pp-renforce}) holds.  Then
\begin{align*}
  n^{1/2-H}  ( \hat\theta_n^{Tap} - \theta) \dlim \;
  \sqrt{\frac{c_1^2\lambda_1^{2H} + \theta^2 c_2^2 \lambda_2^{2H}}{\theta^{-2}
      \lambda_{1} \sigma_{1,e}^2 {\lambda_1} + \lambda_2
      \sigma_{2,e}^2}} \; \frac{\sum_{\ell=1}^m \mathrm {Re} \left(\int_0^1 h_\ell(s) \, \mathrm d
  B(s) \, \int_0^1 h_\ell(t) \, \mathrm d B_H(t) \right)}{\sum_{\ell=1}^m \left|
    \int_0^1 h_\ell(s) \, \mathrm d B(s) \right|^2} \;
\end{align*}
where $B$ is a standard Brownian motion, $B_H$ is a standard fractional Brownian
motion and $B$ and $B_H$ are independent.

\end{theorem}

Since the assumptions of Theorem~\ref{theo:taper-weak} are the same as in
Theorem~\ref{theo:ols-weak}, the comments immediately following that theorem
also applies here.

\subsubsection{Discrete tapered estimator under strong fractional and standard cointegration}

\begin{theorem}
  \label{theo:taper-strong}
  Let Assumptions \ref{hypo:lfgn-dates}, \ref{hypo:lfgn-pp}, \ref{hypo:efficient-shocks}
  \ref{as:independence-shocks} and \ref{hypo:microstructure-strong} hold. Assume
  morover that the counting processes are mutually independent and independent
  of the efficient shocks and~(\ref{eq:pp-renforce-2}) holds.
  \begin{itemize}
  \item If $1/2<H<1$, then
  \begin{align*}
    n^{3/2-H} (\hat \theta_n^{Tap} - \theta) \dlim
    \sqrt{\frac{c_1^2+\theta^2c_2^2}{\theta^{-2} \lambda_{1} \sigma_{1,e}^2 +
        \lambda_2 \sigma_{2,e}^2 } } \; \frac{\sum_{\ell=1}^m \mathrm{Re} \left(\int_0^1 h_\ell(s)
      \, \mathrm d B(s) \, \int_0^1 h'_\ell(s) \, \mathrm d B_H(s) \right) }{
      \sum_{\ell=1}^m |\int_0^1 h_\ell(s) \, \mathrm d B(s)|^2 }
  \end{align*}
  where $B_H$ is a standard fractional Brownian motion independent of the
  standard Brownian motion $B$.
\item If $H=1/2$, $ n (\hat \theta_n^{Tap} -\theta) = O_P(1)$.
\end{itemize}
\end{theorem}

  The assumptions of this theorem are weaker than those of
  Theorem~\ref{theo:ols-strong} on the OLS estimator. The microstructure shocks
  are not assumed to be independent of the counting processes and the efficient
  shocks are not assumed to be Gaussian. Theorem~\ref{theo:ols-strong} can
  presumably be proved without the Gaussian assumption. It might be much more
  difficult in the proof of Theorem~\ref{theo:ols-strong} to avoid the
  assumption of independence between the microstructure shocks and the counting
  processes.

\subsection{A Continuous-Time Tapered Estimator}

\label{sec:Taper-continuous-time}

The estimators of $\theta$ we have considered so far are based on equally-spaced
observations of the log price series. However, under the model (\ref{eq:P1}),
(\ref{eq:P2}), a continuous-time record is available, and it is of interest to
consider using all of the available data to estimate~$\theta$.  Here, for the
sake of theoretical tractability, we consider a tapered estimator $\tilde
\theta$ based on continuously-averaged log prices on adjacent non-overlapping
time intervals. Since the problems with discretization appear only in the strong
fractional and standard cointegration cases, we only consider them in this
section. There is no difference in the case of weak fractional cointegration.

We first establish some notation. Let $\{X(t)\}$ be any time series
defined for all $t \geq 0$, and suppose that we have data on
$\{X_t\}$ for $t \in [0,T]$. Let $\delta>0$ be fixed. In practice, we
might take $\delta$ to be 5 minutes, but the choice of $\delta$ does
not affect the asymptotic distribution we derive below. Define
$ n=\lfloor T/\delta \rfloor$, $\tilde X(0)=0$, and
\[
\tilde X(k) = \int_{u=(k-1)\delta} ^{k \delta} X(u) \, \mathrm d u \; , \ k=1 ,
\cdots ,  n \; .
\]
Then we can define an estimator $\tilde \theta_\delta$ based on these averaged
observations by

\begin{align*}
  \hat \theta_{n,\delta}^{Tap} = \mathrm Re(\tilde \theta_{n,\delta}) \;
\end{align*}
with
\begin{align*}
  \tilde \theta_{n,\delta} = \frac{\sum_{\ell=1}^m d_{\Delta \tilde y_1,\ell}
    \bar d_{\Delta \tilde y_2,\ell}}{\sum_{\ell=1}^m |d_{\Delta \tilde
      y_2,\ell}|^2} \; .
\end{align*}

\subsubsection{Continuous-time tapered estimator under strong fractional and
  standard cointegration}

\begin{theorem}
  \label{theo:strong-continuous-taper}
  Let Assumptions \ref{hypo:lfgn-dates}, \ref{hypo:lfgn-pp},
  \ref{hypo:efficient-shocks}, \ref{as:independence-shocks} and
  \ref{hypo:microstructure-strong-taper-continuous} hold. Assume morover that
  the counting processes are mutually independent and independent of the
  efficient shocks and~(\ref{eq:pp-renforce-2}) holds.
  \begin{itemize}
  \item If $1/2<H<1$, then
  \begin{align*}
    n^{3/2-H} (\hat \theta_{n,\delta}^{Tap} - \theta) \dlim
    \sqrt{\frac{\delta^{2H}(c_1^2+\theta^2c_2^2)}{\theta^{-2} \lambda_{1}
        \sigma_{1,e}^2 + \lambda_2 \sigma_{2,e}^2} } \frac{\sum_{\ell=1}^m
      \mathrm {Re} \left( \int_0^1 h_\ell(s) \, \mathrm d B(s) \, \int_0^1
        h'_\ell(s) \mathrm d B_H( s) \right) }{ \; \sum_{\ell=1}^m |\int_0^1
      h_\ell(s) \, \mathrm d B(s)|^2 }
  \end{align*}
  where $B_H$ is a standard FBM independent of the standard Brownian motion $B$.
\item If $H=1/2$, $ n (\hat \theta_{n,\delta}^{Tap} -\theta) = O_P(1)$.
\end{itemize}
\end{theorem}

Because Assumption \ref{hypo:microstructure-strong-taper-continuous} involves an integral
rather than a sum, we are able to verify that it holds for certain models with
noninteger durations such as ACD and LMSD under certain relationships with the
microstructure shocks.

In Theorem~\ref{theo:strong-continuous-taper}, we allow for leverage effects,
and therefore care is required in defining standard cointegration. As
demonstrated in Lemma~\ref{lem:propogation} (which assumes LMSD durations) if
there is a leverage effect, even when the microstructure shocks are the
differences of a weakly-dependent sequence, the cointegrating error need not be
$I(0)$. In such a case we have strong fractional cointegration rather than the
standard cointegration which might have been anticipated.

It is also possible that even though a leverage effect exists, the memory of
durations has no effect on the degree of cointegration. Specifically, if in
Lemma~\ref{lem:propogation} we replace $\xi_k=Y_{k+1}^2-1$ by $\xi_k =
H_2(Y_{k+1}) - .75H_3(Y_{k+1})$, where $H_2(y) = y^2-1$ and $H_3(y) = y^3-3y$
(the second and third Hermite polynomials, respectively), then there is a
leverage effect with $\mathrm {corr}(\tau_{k+1},\xi_k) = .082$. Nevertheless it
follows from an argument similar to the proof of Lemma~\ref{lem:propogation}
that Assumption~\ref{hypo:microstructure-strong-taper-continuous} holds in this
example with $H=1/2$, so that we have standard cointegration and
Theorem~\ref{theo:strong-continuous-taper} holds with $H=1/2$.

Lemma~\ref{lem:standard-dependence-noleverage} provides an example of standard
cointegration allowing for both time deformation and dependence between the
counting processes and microstructure
shocks. Theorem~\ref{theo:strong-continuous-taper} would hold for this example
with $H=1/2$.

\subsection{Spurious Regressions}
\label{sec:spurious}

In this subsection only, we consider a non-cointegrated version of the model
defined by (\ref{eq:P1}) and (\ref{eq:P2}),
\begin{align}
  y_1(t) & = \sum_{k=1}^{N_1(t)}(e_{1,k}+\eta_{1,k})+\theta_{21}
  \sum_{k=1}^{N_2(t_{1,N_1(t)})} e_{2,k}, \label{eq:P1NoCoint}
  \\
  y_2(t) & = \sum_{k=1}^{N_2(t)}(e_{2,k}+\eta_{2,k})+\theta_{12}
  \sum_{k=1}^{N_1(t_{2,N_2(t)})} e_{1,k}, \label{eq:P2NoCoint}
\end{align}
where $\theta_{12} \neq \theta_{21}^{-1}$. We examine here the properties of the
OLS estimator in the (spurious) regression of $y_1$ on $y_2$ in discrete time
and then briefly discuss corresponding tests for the null hypothesis of
cointegration. Corollary \ref{cor:fclt-spurious} below follows directly from the
proof of Theorem \ref{theo:fclt-levels}.

\begin{corollary}
  \label{cor:fclt-spurious}
  If Assumptions~\ref{hypo:lfgn-dates},
  \ref{hypo:lfgn-pp},~\ref{hypo:efficient-shocks} and \ref{hypo:microstructure}
  are satisfied and $y=(y_1,y_2)$ is given by (\ref{eq:P1NoCoint}) and
  (\ref{eq:P2NoCoint}) with $\theta_{12} \neq \theta_{21}^{-1}$, then as
  $n\to \infty$,
\[
\left(\frac{1}{\sqrt{n}}y(nu)\colon u\in[0,1]\right) \func
B_y=(B_y(u)\colon u\in[0,1]),
\]
where $B_y$ is a bivariate Brownian motion with $2\times 2$ covariance matrix
$\Sigma=(\Sigma_{i,j}\colon i,j=1,2)$ given by the entries
\begin{gather*}
  \Sigma_{1,1} = \lambda_1 \sigma_{1,e}^2 + \theta_{21}^2\lambda_2
  \sigma_{2,e}^2 \; , \ \Sigma_{2,2} = \theta_{12}^2
  \lambda_{1}   \sigma_{1,e}^2 + \lambda_2 \sigma_{2,e}^2 \; , \\
  \Sigma_{1,2} = \theta_{12} \lambda_1 \,  \sigma_{1,e}^2 +
  \theta_{21} \lambda_2  \, \sigma_{2,e}^2 = \Sigma_{2,1}.
\end{gather*}

\end{corollary}

Next, we consider the discretization of $y_1(t)$ and $y_2(t)$ given by
(\ref{eq:P1NoCoint}) and (\ref{eq:P2NoCoint}) at integer time values,

\begin{align}
  y_{1,j}&=\sum_{k=1}^{N_1(j)}(e_{1,k}+\eta_{1,k})+\theta_{21}
  \sum_{k=1}^{N_2(t_{1,N_1(j)})} e_{2,k}, \label{P1jNoCoint}
  \\
  y_{2,j}&=\sum_{k=1}^{N_2(j)}(e_{2,k}+\eta_{2,k})+\theta_{12}
  \sum_{k=1}^{N_1(t_{2,N_2(j)})} e_{1,k}. \label{P2jNoCoint}
\end{align}
Regressing $y_{1,1},\ldots,y_{1,n}$ on $y_{2,1},\ldots,y_{2,n}$ without
intercept, we obtain the OLS estimator
\begin{equation}
  \label{spuriousestimator}
  \hat\delta_n=\frac{\sum_{j=1}^ny_{2,j}y_{1,j}}{\sum_{j=1}^ny_{2,j}^2}.
\end{equation}
Corollary \ref{cor:2} below follows directly from Corollary \ref{cor:fclt-spurious} and the
Continuous Mapping Theorem.

\begin{corollary}
  \label{cor:2}
  If Assumptions~\ref{hypo:lfgn-dates},
  \ref{hypo:lfgn-pp},~\ref{hypo:efficient-shocks} and \ref{hypo:microstructure}
  are satisfied and $y=(y_1,y_2)$ is given by (\ref{eq:P1NoCoint}) and
  (\ref{eq:P2NoCoint}) with $\theta_{12} \neq \theta_{21}^{-1}$, then as
  $n\to \infty$,
\[
\hat \delta_n \dlim \frac{\int_0^1 B_{2,y}(u) B_{1,y}(u) \, du}{\int_0^1 B_{1,y}^2(u) \, du} \,\,\,,
\]
where $B_y=(B_{1,y},B_{2,y})$ is the bivariate Brownian motion given
in Corollary \ref{cor:fclt-spurious}.
\end{corollary}

Corollary \ref{cor:2} together with Corollary \ref{cor:fclt-spurious} can be
used to motivate tests for the null hypothesis of no cointegration. We do not
pursue the details here, but it seems clear that the null distribution for unit
root tests based on the residuals $\{y_{1,j}-\hat\delta_n y_{2,j}\}_{j=1}^n$ can
be derived from Corollaries \ref{cor:fclt-spurious} and \ref{cor:2}, and that
these null distributions will have form similar to the distributions listed, for
example, in \citet[Proposition 19.4]{MR1278033}.

\section{Proofs}
\label{sec:proofs}

\begin{proof}[Proof of Lemma \ref{lem:TimeDeformation-nonstationaire}]
Write
\[
\frac{N(t)}t = \frac{\tilde N(f(t))}t = \frac{\tilde N(f(t))}{f(t)}
\frac{f(t)}{t} \;.
\]
By assumption, $f(t)\mapsto\infty$ (in probability if $f$ is random), thus $\tilde
N(f(t))/f(t)$ converges in probability to $\tilde\lambda$. By assumption, it
also holds that $t^{-1}f(t)\plim\gamma$. Thus, $\frac{N(t)}{t}
\plim \tilde \lambda\gamma$, so that
Assumption~\ref{hypo:lfgn-pp} holds for $N$ with $\lambda =
\tilde\lambda\gamma$. Next, we note that $N(t_{k}^-) \leq k$, thus
\begin{align*}
  1 & \leq \frac{N(t_{k})} k = 1 + \frac{N(t_{k}) - k}k \leq 1 +
  \frac{N(t_{k}) - N(t_{k}^-)}k \\
  & = 1 + \frac{\tilde N(f(t_{k}^-),f(t_{k})]}k \leq 1 +
  \frac{\tilde N(f(t_k)-C,f(t_{k})]}{f(t_k)} \frac{f(t_k)}k
  \end{align*}
  using the definition of $N$ and the boundedness requirement on the jumps of
  $f$.  Since $f(t_k)$ tends to infinity, it suffices to prove that if
  Assumptions~\ref{hypo:lfgn-dates} and~\ref{hypo:lfgn-dates} hold for $\tilde
  N$, then for any fixed positive $C$, it holds that
  \begin{align*}
    \frac{\tilde N(t-C,t]}{t} \plim 0 \; .
  \end{align*}
  Without loss of generality, set $\tilde\lambda=1$. Fix some
  $\epsilon\in(0,1/2)$.  Since $\tilde t_k/k\plim1$, with probability
  arbitrarily close to~1 (say bigger than $1-\epsilon$), there exists $k_0$ such
  that $\tilde t_k/k\in[1-\epsilon,1+\epsilon]$ for all $k\geq k_0$. For $k\geq
  k_0$, if $t_k \in(t-C,t]$, it necessarily holds that $t-C \leq k(1+\epsilon)$
  and $k(1-\epsilon)\leq t$. Hence, $(t-C)/(1+\epsilon) \leq k \leq
  t/(1-\epsilon)$. This implies that for $t$ large enough
  ($t<C+k_0(1+\epsilon)$), $\tilde N(t-C,t)$ is less than the number of integers
  between $(t-C)/(1+\epsilon)$ and $t/(1-\epsilon)$, i.e. $\tilde N(t-C,t] \leq
  c \epsilon t$, for some constant $c>0$. Thus $\tilde N(t)/t \leq c\epsilon$
  with probability tending to~1, and this proves that $\tilde N(t-C,t]/t\plim0$.
  It follows that $N(t_{k})/k = 1 + o_P(1)$, i.e.  $N(t_k)/t_{k}$ converges in
  probability to 1. Thus
\[
\frac{t_{k}}{k} = \frac{t_{k}}{N(t_{k})} \frac{N(t_{k})}{k}
\plim \frac{1}{\lambda} \; .
\]
\end{proof}

\begin{proof}[Proof of Lemma~\ref{lem:propogation}]
  Denote $H_\tau$ by $H$ to simplify the notation. Define $\xi_k =
  Y_{k+1}^2-1$. Then $\xi_k$ is centered, has finite variance summable
  autocovariance function, since $\mathrm{cov}(\xi_0,\xi_k) =
  2\mathrm{cov}^2(Y_0,Y_{k+1})$. Thus $\{\xi_k\}$ has a summable autocovariance
  function because $H\in(1/2,3/4)$. By \citet[Theorem~4]{MR1331224}, this implies
  that $\{\xi_k\}$ is in the domain of attraction of the standard Brownian
  motion, i.e.
\begin{align*}
  n^{-1/2} \sum_{k=1}^{[n\cdot]} \xi_k \func c' B \; ,
\end{align*}
with ${c'}^2 = \mathrm{var}(\xi_0) + 2 \sum_{k=1}^\infty
\mathrm{cov}(\xi_0,\xi_k)$. This proves~(\ref{eq:weak-dep}).

Assume now that $\tau_k = \epsilon_k \mathrm e^{\sigma Y_k}$ (with $\sigma=1$ in
the statement of the Lemma). The properties of Hermite polynomials yield that
$\esp[\mathrm e^{\sigma Y_0} H_j(Y_0)] = \sigma^j \mathrm e^{\sigma^2/2}$ for
all $j\geq1$.  Denote now $\lambda^{-1} = \esp[\tau_k] = \esp[\mathrm e^{\sigma
  Y_k}] = \mathrm e^{\sigma^2/2}$, $m=\esp[\xi_{k-1}\tau_k] = \esp[(Y_k^2-1)
\mathrm e^{Y_k}] = \sigma^2\mathrm e^{\sigma^2/2}$ and $G(y) = (y^2-1) \mathrm
e^{\sigma y} - m$. We now prove that~(\ref{eq:lrd-random-index}) holds with
$\mu^*=\lambda m$. Write
\begin{align}
  \int_0^{T} & (\xi_{N(s)} - \lambda m) \, \mathrm d s  \nonumber \\
  & = t_1 \xi_0 + \sum_{k=1}^{N(T)} \tau_{k+1} \xi_k - \lambda m T +
  (t_{N(T)+1}-T)  \xi_{N(T)+1}  \nonumber \\
  & = t_1 \xi_0 + \sum_{k=1}^{N(T)} (\epsilon_{k+1}-1) \xi_{k} \mathrm e^{\sigma
    Y_{k+1}} + \sum_{k=0}^{N(T)} G(Y_{k+1}) \label{eq:decomp} \\
  & + m(N(T)-\lambda T) - (t_{N(T)+1}-T) \xi_{N(T)+1} \; . \nonumber 
  \end{align}
  By Lemma~\ref{lem:check-lmsd} and applying H\"older's inequality, it can be
  shown that $(t_{N(T)+1}-T)(\xi_{N(T)+1}-\rho) = O_P(1)$. Since the sequence
  $\{\epsilon_k\}$ is independent of the Gaussian process $\{Y_k\}$, the second
  term in the righthand side of~(\ref{eq:decomp}) is in the domain of
  attraction of the standard Brownian motion, and the normalizing sequence is
  $\sqrt n$.  Thus we must obtain the joint asymptotic behaviour of
  $\sum_{k=1}^{N(Tt)} G(Y_k)$ and $N(Tt)-\lambda Tt$.

  The durations are in the domain of attraction of the fractional Brownian
  motion with Hurst index $H$, since
  \begin{align*}
    \sum_{k=1}^ n (\tau_{k} - \lambda^{-1}) & = \sum_{k=1}^n (\epsilon_k-1)
    \mathrm e^{\sigma Y_k} + \sum_{k=1}^n (\mathrm e^{\sigma Y_k} - \lambda^{-1}) \; .
  \end{align*}
  The first term in the righthand side is $O_P(\sqrt n)$ and the second sum,
  suitably normalized converges to the fractional Brownian motion with Hurst
  index $H$ because the function $x\mapsto\mathrm e^{\sigma x}-\lambda^{-1}$ has
  Hermite rank 1. See e.g. \cite{MR1331224}. More precisely, let
  $c_1=\esp[Y_1\mathrm e^{\sigma Y_1}]=\sigma\mathrm e^{\sigma^2/2}$ and define $g(y) =
  \mathrm e^{\sigma y} - \lambda^{-1} - c_1 y$.  The function $g$ has Hermite
  rank 2, and since $H\in(1/2,3/4)$, this implies that
  \begin{align*}
    \mathrm {var}\left(\sum_{k=1}^n g(Y_k) \right) = O(n) \; .
  \end{align*}
  Thus $\sum_{k=1}^n (\tau_k-\lambda^{-1})$ is asymptotically equivalent to $c_1
  \sum_{k=1}^n Y_k$.  Let $B_H$ denote the standard fractional Brownian motion
  with hurst index $H$. The assumption on the covariance of the Gaussian process
  $\{y_k\}$ implies that
  \begin{align*}
    n^{-H} \sum_{k=1}^{[n \cdot]} Y_k \func \varphi B_H
  \end{align*}
  with $\varphi^2 = c/\{H(2H-1)\}$.  Denote now $c_2 =
  \esp[Y_1G(Y_1)]=\sigma(\sigma^2+2)\mathrm e^{\sigma^2/2}$ and define $h(y) =
  G(y) - c_2 y$. Then $h$ has Hermite rank 2 and thus by similar arguments as
  above, $\sum_{k=1}^n G(Y_k)$ is asymptotically equivalent to $c_2 \sum_{k=1}^n
  Y_k$.  Thus we obtain
  \begin{align*}
    n^{-H} \left( \sum_{k=1}^{[nt]} (\tau_k - \lambda^{-1}), \sum_{k=1}^{[nt]}
      G(Y_k) \right) \func (c_1 \varphi B_H(t), c_2 \varphi B_H(t)) \; .
  \end{align*}
  By Vervaat's Lemma (see \cite{MR0321164} or \citet[Proposition
  3.3]{MR2271424}), the previous convergence implies that
  \begin{align*}
    n^{-H} \left( N(nt) - \lambda n t, \sum_{k=1}^{[nt]} G(Y_k) \right)
    \func (-\lambda c_1 \varphi B_H(\lambda t), c_2 \varphi B_H(t)) \; .
  \end{align*}
  By the continuity of the composition map, this yields
  \begin{align*}
    n^{-H} \left( N(nt) - \lambda n t, \sum_{k=1}^{N(nt)} G(Y_k) \right)
    \func (-\lambda c_1 \varphi B_H(\lambda t), c_2 \varphi B_H(\lambda
    t)) \; .
  \end{align*}
Next we obtain  that
\begin{align*}
  n^{-H} \left\{ \sum_{k=1}^{N(nt)} G(Y_k) + m( N(nt) - \lambda nt) \right\}
  \func \varphi (c_2 - \lambda m c_1) B_H(\lambda t) \;
\end{align*}
with $c_2-\lambda mc_1 = 2\sigma \mathrm e^{\sigma^2/2}>0$.
We conclude that $n^{-H} \int_0^{n\cdot} \{\xi_{N(s)} - \lambda m\} \, \mathrm d s
\func \varphi (c_2 - \lambda m c_1) B_H$.
\end{proof}

\subsection{Proof of Theorem~\ref{theo:fclt-levels} and Corollary~\ref{cor:fclt-spurious}}

We first need the following Lemma.
\begin{lemma}
   \label{lem:dini}
   Under Assumption~\ref{hypo:lfgn-dates} and~\ref{hypo:lfgn-pp},
   $N_i(t_{j,N_j(nt)})/n$ converges in probability uniformly on compact sets to
   $\lambda_i t$, where $\{i,j\}=\{1,2\}$.
\end{lemma}

\begin{proof}[Proof of Lemma~\ref{lem:dini}]
  The sequence of (random) functions $N_i(n\cdot)/n$ is nondecreasing and
  converges pointwise in probability to $\lambda_i t$ by ergodicity.  A sequence
  of nondecreasing function converging to a continuous function converges
  uniformly on compact sets. This results is known as Dini's
  Theorem. Cf. \citet[page~3]{MR900810}. Thus the convergence of $N_i(n\cdot)/n$
  is uniform on compact sets. Assumptions~\ref{hypo:lfgn-dates}
  and~\ref{hypo:lfgn-pp} imply that $N_i(t)\plim\infty$ and
  $t_{i,n}\plim\infty$. Thus
  \begin{align*}
    \frac{ N_i(t_{j,N_j(nu)}) }{n} = \frac{ N_i(t_{j,N_j(nu)}) } {t_{j,N_j(nu)}}
    \times \frac{t_{j,N_j(nu)}}{N_j(nu)} \times \frac{N_j(nu)}{n} \plim \lambda_i
    \times \frac1{\lambda_j} \times \lambda_j u = \lambda_i u \; .
  \end{align*}
  Applying again Dini's lemma, we also have that $N_i(t_{j,N_j(nu)})/n$
  converges  uniformly on compact sets to $\lambda_iu$.
\end{proof}

\begin{proof}[Proof of Theorem~\ref{theo:fclt-levels}]
  Denote $S_{i,n}^{e}(t) = \sum_{k=1}^{[nt]} e_{i,k}$ and $S_{i,n}^{\eta}(t) =
  \sum_{k=1}^{[nt]} \eta_{i,k}$, $i=1,2$.  Under
  Assumptions~\ref{hypo:efficient-shocks} and~\ref{hypo:microstructure},
  $n^{-1/2}(S_{1,n}^e,S_{2,n}^e,S_{1,n}^\eta,S_{1,n}^\eta)$ converges weakly to
  $(\sigma_{1,e} B_1, \sigma_{2,e} B_2,0,0)$, where $B_1$ and $B_2$ are
  independent standard Brownian motions. This follows from the independence of
  $e_1$ and $e_2$ and the local uniform convergence to 0 in probability of
  $n^{-1/2}S_{i,n}^{\eta}$.  With the previous notation, (\ref{eq:P1NoCoint})
  and (\ref{eq:P2NoCoint}) become
  \begin{align*}
    y_1(nt) & = S_{1,n}^e(N_1(nt)) + \theta_{21} S_{2,n}^e(N_2(t_{1,N_1(nt)})) +
    S_{1,n}^\eta(N_1(nt)) \; ,
    \\
    y_2(nt) & = S_{2,n}^e(N_2(nt)) + \theta_{12} S_{1,n}^e(N_1(t_{2,N_2(nt)})) +
    S_{2,n}^\eta(N_2(nt)) \; .
  \end{align*}
  By Lemma~\ref{lem:dini} and the continuity of the composition map on $\mathcal
  C \times \mathcal C$ endowed with the metric of uniform convergence on compact
  sets (see e.g.  \citet[Chapter 3, Section 17]{MR0233396}), we obtain the joint
  convergence of
  \begin{multline*}
    n^{-1/2} \left( S_{1,n}^e(N_1(n\cdot)), S_{1,n}^e(N_1(t_{2,N_2(n\cdot)})),
    \right.    \\
    \left.  S_{2,n}^e(N_2(n\cdot)), S_{2,n}^e (N_2(t_{1,N_1(n\cdot)})) ,
      S_{1,n}^\eta(N_1(n\cdot)), S_{2,n}^\eta(N_2(n\cdot)) \right)
  \end{multline*}
  to $(\sigma_{1,e} \sqrt{\lambda_1} B_1,\sigma_{1,e}
  \sqrt{\lambda_2}B_1, \sigma_{2,e} \sqrt{\lambda_2} B_2,\sigma_{2,e}
  \sqrt{\lambda_2} B_2, 0, 0)$. This yields
  Corollary~\ref{cor:fclt-spurious} and Theorem~\ref{theo:fclt-levels} by setting
  $\theta_{21}=\theta$ and $\theta_{12} = \theta^{-1}$.
\end{proof}

\subsection{Proof of Theorems~\ref{theo:ols-weak} \ref{theo:t}
  \ref{theo:Aggregation} and~\ref{theo:ols-strong}}

\begin{proof}[Proof of Theorem~\ref{theo:ols-weak}]
Write 
\begin{align*}
  \hat\theta_n^{OLS} = \theta + \frac{\sum_{j=1}^n \{y_1(j)-\theta y_2(j)\}
    y_2(j)}{\sum_{j=1}^n y_2^2(j)} \; .
\end{align*}
Assumptions \ref{hypo:lfgn-dates}, \ref{hypo:lfgn-pp},
\ref{hypo:efficient-shocks}, \ref{as:independence-shocks} and
\ref{hypo:microstructure-weak} imply those of
Theorem~\ref{theo:fclt-levels}. Thus we can apply the Continuous Mapping Theorem
and obtain
\begin{align}
  \label{eq:denominator}
  n^{-2} \sum_{j=1}^n y_2^2(j) \dlim \{\theta^{-2} \lambda_{1}
  \sigma_{1,e}^2 + \lambda_2 \sigma_{2,e}^2 \} \int_0^1B^2(s) \, \mathrm d s \; ,
\end{align}
where $B$ is a standard Brownian motion.  Thus, in order to study the
convergence of $\hat\theta_n^{OLS} - \theta$ suitably renormalized, it suffices to
study the sum
\begin{align*}
  \sum_{j=1}^n \{y_1(j)-\theta y_2(j)\} y_2(j) \; .
\end{align*}
We further decompose the cointegrating error.  Denote
\begin{align*}
  y_1^e(j) & = \sum_{k=1}^{N_1(j)} e_{1,k} + \theta
  \sum_{k=1}^{N_2(t_{1,N_1(j)})} e_{2,k} \; , \ \ \ y_1^\eta(j) =
  \sum_{k=1}^{N_1(j)} \eta_{1,k} \; , \\
  y_2^e(j) & = \sum_{k=1}^{N_2(j)} e_{2,k} + \theta^{-1}
  \sum_{k=1}^{N_1(t_{2,N_2(j)})} e_{1,k} \; , \ \ \ y_2^\eta(j) =
  \sum_{k=1}^{N_2(j)} \eta_{2,k} \; , \\
  r_{1,j} & = \sum_{k=N_1(t_{2,N_2(j)})+1}^{N_1(j)} e_{1,k} \; , \ \ \ r_{2,j}
  = \sum_{k=N_2(t_{1,N_1(j)})+1}^{N_2(j)} e_{2,k} \; .
\end{align*}
With this notation, we can write
\begin{align}
  \sum_{j=1}^n \{y_1(j)-\theta y_2(j)\} y_2(j) & = \sum_{j=1}^n \{r_{1,j} -
  \theta r_{2,j}\} y_2(j) + \sum_{j=1}^n \{y_1^\eta(j)-\theta y_2^\eta(j)\}
  y_2(j) \; . \label{eq:cointegrating-error}
\end{align}
Applying Theorem~\ref{theo:fclt-levels},
Assumption~\ref{hypo:microstructure-weak} and the Continuous Mapping Theorem, we
obtain
\begin{align*}
  n^{-3/2-H} &  \sum_{j=1}^n \{y_1^\eta(j) - \theta y_2^\eta(j)\} y_2(j) \\
  & \dlim \int_0^1 \{\theta^{-1}\sqrt{\lambda_1}\sigma_{1,e}B_1(t) + \sqrt{\lambda_2}
  \sigma_{2,e} B_2(t)\} \{ c_1 B_{1,H}(\lambda_1t) - \theta c_2
  B_{2,H}(\lambda_2t) \}\, \mathrm d t \\
  & \equallaw \Sigma \int_0^1 B(t) B_H(t) \, \mathrm d t
\end{align*}
    where $B$ is a standard Brownian motion, $B_H$ is a fractional Brownian
    motion, independent of $B$ and
\begin{align}
  \label{eq:var-frac}
  \Sigma^2 = (\theta^{-2} \lambda_{1} \sigma_{1,e}^2 + \lambda_2 \sigma_{2,e}^2)
  (c_1^2 \lambda_1^{2H} + \theta^2 c_2^2 \lambda_2^{2H}) \; .
\end{align}
There only remains to prove that, for $i=1,2$,
  \begin{align}
    \label{eq:OP(n3/2)}
      n^{-3/2} \sum_{j=1}^n r_{i,j} y_2(j) = O_P(1) \; .
  \end{align}
  The convergence of $n^{-1/2}y_2$ is uniform on $[0,1]$, so $n^{-1/2}\max_{1\leq
    j \leq n} |y_2(j)| = O_P(1)$.  Therefore, it suffices to prove that
  \begin{align}
  \label{eq:Op1}
    n^{-1} \sum_{j=1}^n |r_{i,j}|  = O_P(1) \; .
  \end{align}
Recall that  $N_i(s) < k \Leftrightarrow
  t_{i,k} > s$. Thus, for $k\leq N_1(n)$,
  \begin{align*}
    N_1(t_{2,N_2(j)}) < k \leq N_1(j) \Leftrightarrow t_{2,N_2(j)} < t_{1,k} \leq j \; .
  \end{align*}
  The first inequality on the righthand side means that there is no point of
  $N_2$ between $t_{1,k}$ and $j$, i.e. $j \leq t_{2,N_2(t_{1,k})+1}$. Let
  $A_2(t) = t_{2,N_2(t)+1}-t$ denote the forward recurrence time of $N_2$,
  i.e. the time between $t$ and the next event of $N_2$ after $t$.  Thus,
  \begin{align*}
    \sum_{j=1}^n |r_{1,j}| \leq \sum_{j=1}^n
    \sum_{k=N_1(t_{2,N_2(j)})+1}^{N_1(j)} |e_{1,k}| = \sum_{k=1}^{N_1(n)}
    |e_{1,k}| \{A_2(t_{1,k}) +1\}\; .
  \end{align*}
  We thus get the bound for the conditional expectation given the sigma-field
  $\mathcal N$ generated by the counting processes $N_1$ and $N_2$:
  \begin{align*}
    \esp \left[ \sum_{j=1}^n |r_{1,j}| \mid \mathcal N \right] \leq C
    \sum_{k=1}^{N_1(n)} A_2(t_{1,k}) \; .
  \end{align*}
  Conditioning on $N_1$ and applying \eqref{eq:pp-renforce} yields
\begin{align*}
  \esp \left[ \sum_{j=1}^n |r_{1,j}| \mid N_1 \right] \leq C N_1(n) = O_P(n) \; .
\end{align*}
This proves~(\ref{eq:Op1}) and concludes the proof of
Theorem~\ref{theo:ols-weak}.
\end{proof}

\begin{proof}[Proof of Theorem~\ref{theo:t}]
We note first that
\begin{align*}
  n \hat \sigma^2_{\hat\theta_n^{OLS}} & =
  \frac{\sum_{j=1}^n[y_1(j)-\hat\theta_n^{OLS}y_2(j)]^2}{\sum_{j=1}^n y_2^2
    (j)}= \frac{\sum_{j=1}^n[y_1(j)-\theta y_2(j)
    +(\theta-\hat\theta_n^{OLS})y_2(j)]^2}{\sum_{j=1}^n y_2^2 (j)} \\
  & = \frac{\sum_{j=1}^n[y_1(j)-\theta y_2(j)]^2} {\sum_{j=1}^n y_2^2 (j)} -
  (\theta-\hat\theta_n^{OLS})^2 \; .
\end{align*}
Thus,
\[
\frac{n}{t_n^2} = \frac{1}{n^{1-2H}(\hat\theta_n^{OLS}-\theta)^2}
\frac{n^{-2H-1}\sum_{j=1}^n [y_1(j)-\theta y_2(j)]^2}{n^{-2}\sum_{j=1}^n
  y_2^2(j)} -1 \; .
\]
Note that
\begin{align*}
  n^{-2H-1} & \sum_{j=1}^n [y_1(j)-\theta y_2(j)]^2 \\
  & = n^{-2H-1} \sum_{j=1}^n [y_1^\eta (j) - \theta y_2^\eta (j)]^2 \\
  & \ \ \ + n^{-2H-1} \sum_{j=1}^n (r_{1,j} - \theta r_{2,j} )^2 + 2 n^{-2H-1}
  \sum_{j=1}^n (r_{1,j} - \theta r_{2,j} ) (y_1^\eta (j) - \theta y_2^\eta (j)) \; .
\end{align*}
By Theorem~\ref{theo:fclt-levels} and the Continuous Mapping Theorem, the first
term is $O_P(1)$. We will prove below that $n^{-1} \sum_{j=1}^n (r_{1,j}-\theta
r_{2,j})^2 = O_P(1)$, implying that the second and last terms are $o_P(1)$.  By
the proof of Theorem~\ref{theo:ols-weak}, we also have that
\begin{align*}
  n^{1/2-H} (\hat\theta_n^{OLS} - \theta) = \frac{n^{-3/2-H} \sum_{j=1}^n
    (y_1^\eta(j)-\theta y_2^\eta(j))y_2(j) + o_P(1)}{n^{-2} \sum_{j=1}^n y_2^2(j)} \; .
\end{align*}
Thus, we can write
\begin{align*}
  \frac n{t_n^2} & = \frac {\left\{n^{-2} \sum_{j=1}^n y_2^2(j)\right\}
    \left\{n^{-2H-1}\sum_{j=1}^n [y_1^\eta(j)-\theta y_2^\eta(j)]^2 +
      o_P(1)\right\}} { \left\{n^{-3/2-H} \sum_{j=1}^n [y_1^\eta(j)-\theta
      y_2^\eta(j)] y_2(j) + o_P(1)\right\}^2} - 1 \; .
\end{align*}
By Theorem~\ref{theo:fclt-levels}, we know that $(n^{-H-1} \sum_{j=1}^{[n\cdot]}
[y_1^\eta(j)-\theta y_2^\eta(j)],n^{-1/2} y_2([n\cdot]))$ converge jointly to
$(\varsigma B,\varsigma_HB_H)$, where $B$ is a standard Brownian motion, $B_H$
is a standard fractional motion, mutually independent, and $\varsigma$ and
$\varsigma_H$ are positive constants. Thus, by the Continuous Mapping Theorem,
we have
\[
\frac{n}{t_n^2} \dlim \frac{\int_0^1 B^2(t) dt \int_0^1 B_H^2(t)dt} {\left \{
    \int_0^1 B(t) B_H(t)dt\right \}^2} - 1 \; .
\]
We now deal with the remainder term $n^{-1} \sum_{j=1}^n (r_{1,j}-\theta
r_{2,j})^2$. We only prove that $n^{-1} \sum_{j=1}^n r_{1,j}^2 = O_P(1)$, the
proof for the term involving $r_{2,j}$ being similar. Since the counting
processes and the efficient shocks are independent, taking conditional
expectations, we obtain
\begin{align*}
  \esp \left[ \sum_{j=1}^n r_{1,j}^2 \mid \mathcal N \right] = \sum_{j=1}^n
  \sum_{k=N_1(t_{2,N_2(j)})+1}^{N_1(j)} \sigma_{1,e}^2 = \sigma_{1,e}^2
  \sum_{k=1}^{N(n)} \{A_2(t_{1,k}) +1\} \; .
\end{align*}
Sicne the counting processes are mutually independent, we can apply
Condition~(\ref{eq:pp-renforce}) to see that the expectation of the sum in the
righthand side is $O(n)$. Thus $\sum_{j=1}^n r_{1,j}^2=O_P(n)$ and this
concludes the proof.
\end{proof}

\begin{proof}[Proof of Theorem~\ref{theo:Aggregation}]
  We will prove below that there exists a positive constant $C$ such that
  \begin{align} 
    \label{eq:SmConsistency}
    m^{2-2H} s_m^2 \plim C \; .
  \end{align}
  Thus,
\[
\log (m^{2-2H} s_m^2) \plim \log C \; , 
\]
and
\[
\hat H = 1 + \frac{\log(s_m^2)}{2\log m} = H + \frac{\log(m^{2-2H}s_m^2)}{2\log
  m} \plim H \; .
\]
\end{proof}

\begin{proof}[Proof of~\eqref{eq:SmConsistency}]
Elementary algebra yields
\[
m X_k^{(m)} = A_{k,m} + V_{k,m} - R_{k,m}
\]
where
\begin{align*}
  A_{k,m}  = \sigma_1\{B_{1,H} (N_1(km)) & - B_{1,H} (N_1((k-1)m))\} \\
  & - \theta\sigma_2\{B_{2,H} (N_2(km)) - B_{2,H} (N_2((k-1)m))\} \; ,
  \\
  R_{k,m}  = (\hat \theta_n^{OLS}  - \theta ) [y_2(km) & - y_2((k-1)m)] \; ,
\end{align*}
\begin{align*}
  V_{k,m} = \sum_{\ell = N_1(t_{2,N_2(km)})+1}^{N_1(km)}e_{1,\ell} & - \theta
  \sum_{\ell = N_2(t_{1,N_1(km)})+1}^{N_2(km)} e_{2,\ell} \\
  & - \sum_{\ell = N_1(t_{2,N_2((k-1)m)})+1}^{N_1((k-1)m)} e_{1,\ell} + \theta
  \sum_{\ell = N_2(t_{1,N_1((k-1)m)})+1}^{N_2((k-1)m)} e_{2,\ell} \; .
\end{align*}
The convergence~\eqref{eq:SmConsistency} is a consequence of the following three
convergences. 
\begin{align}
  \label{eq:RNegligible}
  m^{1-2H} n^{-1} \sum_{k=1}^{\lfloor n/m \rfloor} R_{k,m}^2
  \plim 0 \; , \\
  \label{eq:VNegligible}
 m^{1-2H}  n^{-1} \sum_{k=1}^{\lfloor n/m \rfloor}
 V_{k,m}^2 \plim 0 \; , \\
 \label{eq:AConsistency}
 m^{1-2H} n^{-1} \sum_{k=1}^{\lfloor n/m \rfloor} A_{k,m}^2 \plim C \; .
\end{align}
We will only prove~\eqref{eq:AConsistency}, the other convergences being similarly
and more easily obtained. Let $T_n = m^{1-2H} n^{-1} \sum_{k=1}^{\lfloor n/m
  \rfloor} A_{k,m}^2$ and let $\mathcal N$ denote the sigma-field generated by
the point processes $N_1$ and $N_2$.  We will prove that there exists a positive constant
$C$ such that
\begin{align}
  \label{eq:conv-prob-esp-con}
  & \lim_{n\to\infty} \esp[T_n\mid \mathcal N] = C \; , \\
  & \lim_{n\to\infty} \var(T_n \mid \mathcal N) = 0 \; . \label{eq:conv-var-con}
\end{align}
%
%
By the Bienaym\'e-Chebyshev inequality, (\ref{eq:conv-prob-esp-con})
and~(\ref{eq:conv-var-con}), we have
\begin{align*}
  \pr(|T_n-C| > \epsilon \mid \mathcal N) & \leq \epsilon^{-2} \esp[|T_n-C|^2 \mid
  \mathcal N]  \\
  & \leq \epsilon^{-2} \var(T_n \mid \mathcal N) + \epsilon^{-2} (\esp[T_n\mid
  \mathcal N]-C)^2 \plim 0 \; .
\end{align*}
This precisely means that $T_n$ converges to $C$ in conditional probability,
i.e.~for all $\epsilon>0$,
\begin{align*}
  \lim_{n\to\infty} \pr(|T_n-C| > \epsilon \mid \mathcal N) = 0 \; .
\end{align*}
Since a probability is bounded by one and $\pr(|T_n-C|>\epsilon) =
\esp[\pr(|T_n-C|>\epsilon \mid \mathcal N)]$, the bounded convergence theorem
implies that for all $\epsilon>0$,
\begin{align*}
  \lim_{n\to\infty} \pr(|T_n-C| > \epsilon) = 0 \; , 
\end{align*}
i.e. $T_n\plim C$.

For simplicity of notation, we also assume that $n/m$ is an integer. For any
$a<b$, $N_i(a,b]$ denotes the number of points of $N_i$ in the interval
$(a,b]$. Note that $N_i(0,t]=N_i(t)$ for all~$t>0$.  Since $B_{1,H}$ and $B_{2,H}$
are independent of $\mathcal N$, we have, for any $s<t$,
\begin{align}
  \label{eq:varcon-BH}
  \var( B_{i,H}(N_i(t)-B_{i,H}(N_i(s)) \mid \mathcal N) = \{N_i(s,t]\}^{2H} \; .
\end{align}
Since moreover $B_{1,H}$ and $B_{2,H}$ are mutually independent, this yields
\begin{align*}
  \esp&[T_n \mid \mathcal N] \\
& = \sigma_1^2 m^{1-2H} n^{-1} \sum_{k=1}^{n/m}
  \{N_1(((k-1)m,km])\}^{2H} + \sigma_2^2 \theta^2 m^{1-2H} n^{-1}
  \sum_{k=1}^{n/m} \{N_2(((k-1)m,km])\}^{2H}
  \\
  & = \sigma_1^2 T_{1,n} + \sigma_2^2 \theta^2T_{2,n} \; .
\end{align*}
We will prove that $T_{i,n}$ converges in probability to $\lambda_i^{2H}$,
$i=1,2$, where $\lambda_i$ is the intensity of $N_i$,
i.e.~$\esp[N_i(0,1]]=\lambda_i$.  This will imply~(\ref{eq:conv-prob-esp-con}) with
$C = \sigma_2^2 \lambda_1^{2H}+\sigma_2^2 \theta^2\lambda_2^{2H}$.

Since $N_i$ is stationary, we have
\begin{align*}
  \esp \left[ \left| T_{i,n} - \lambda_i^{2H} \right| \right] \leq \esp \left[
    \left| \{m^{-1} N_i(m)\}^{2H} - \lambda_i^{2H} \right| \right] \; .
\end{align*}
For brevity, we now omit the subscript $i$. By stationarity and ergodicity,
$m^{-1}N(m)$ converges almost surely to $\lambda$ as $m$ goes to infinity. Since
$0<2H<1$, it holds that $|a^{2H}-b^{2H}|\leq |a-b|^{2H}$ for all real numbers
$a$, $b$. Thus, for any $\epsilon>0$,
\begin{align*}
  \esp & \left[ \left| \{m^{-1} N(m)\}^{2H} - \lambda^{2H} \right| \right] \\
  & \leq \epsilon^{2H} + \esp \left[ \left| \{m^{-1} N(m)\}^{2H} - \lambda^{2H}
    \right| \mathbf1_{|m^{-1} N(m) - \lambda|>\epsilon\}} \right] \\
  & \leq \epsilon^{2H} + \esp \left[\{m^{-1} N(m)\}^{2H} \mathbf1_{|m^{-1}
      N(m) - \lambda|>\epsilon\}} \right] +\lambda^{2H} \pr(|m^{-1} N(m) -
  \lambda|>\epsilon)\; .
\end{align*}
By ergodicity, the last term above tends to zero as $m$ tends to infinity. To
deal with the middle term, we apply H\"older's inequality and obtain
\begin{align*}
  \esp  [\{m^{-1} N(m)\}^{2H} & \mathbf1_{|m^{-1} N(m) -
      \lambda|>\epsilon\}}] \\
  & \leq \esp^{2H} [m^{-1} N(m) ]
  \pr^{1-2H}(|m^{-1} N(m) - \lambda|>\epsilon)  \\
  & = \lambda^{2H} \pr^{1-2H}(|m^{-1} N(m) - \lambda|>\epsilon) \to 0  \; ,
\end{align*}
as $m\to\infty$, again by ergodicity. Thus we obtain that $\lim_{n\to_infty}
\esp \left[ \left| T_{i,n} - \lambda_i^{2H} \right| \right]=$ and this concludes
the proof of~(\ref{eq:conv-prob-esp-con}).

We now prove~(\ref{eq:conv-var-con}). We denote the conditional variance and
covariance given $\mathcal N$ by $\var_N$ and $\cov_N$, respectively. Since $A_{k,m}$ is
conditionally Gaussian, we have, for all $k,k'$,
\begin{align*}
  \cov_N(A_{k,m}^2,A_{k',m}^2)) = 2 \cov_N^2(A_{k,m},A_{k',m}) \; .
\end{align*}
Thus, denoting $S_n = \sum_{k=1}^{[n/m]} A_{k,m}^2$, we have
\begin{align*}
  \var_N(S_n) & = 2 \sum_{k=1}^{n/m} \var_N^2(A_{k,m}) + 4 \sum_{k=1}^{n/m-1}
  \sum_{k'=k+1}^{n/m} \cov^2_N(A_{k,m},A_{k',n}) = 2 \times I + 4 \times II \; .
\end{align*}
Applying~(\ref{eq:varcon-BH}) to compute $\var_N(A_{k,m})$ and taking
expectation, we have, by stationarity,
\begin{align*}
  m^{2-4H} n^{-2} \esp[I] = \sigma_1^2 mn^{-1} \esp[(N_1(m)/m)^{4H}] + 
  \sigma_2^2 \theta^2 mn^{-1}\esp[(N_2(m)/m)^{4H}] \; .
\end{align*}
Since $4H<2$, by Jensen's inequality, $\esp[N_i^{4H}(m)] \leq
\{\esp[N_i^2(m)]\}^{2H}$ and by stationarity $\esp[N_i^2(m)] \leq
m^2\esp[N_i^2(1)]$. Thus $m^{2-4H} n^{-2} \esp[I] = O(m/n)$.

  Consider now the last term $II$. For any positive real
numbers $s<t<u<v$ and a standard fractional Brownian motion $H$, we have
\begin{align*}
  \cov(B_H(t) - B_H(s),B_H(v)-B_H(u)) & = |v-s|^{2H} - |v-t|^{2H} + |u-t|^{2H} -
  |u-s|^{2H} \; .
\end{align*}
Thus, 
\begin{align*}
  \cov_N(A_{k,m},A_{k',n}) & = \sigma_1^2\left\{N_1^{2H}((k-1)m,k'm] - N_1^{2H}(km,k'm] \right. \\
  & + \left. N_1^{2H}(km,(k'-1)m] - N_1^{2H}((k-1)m,(k'-1)m] \right\} \\
& + \sigma_2^2\theta^2\left\{N_2^{2H}((k-1)m,k'm] - N_2^{2H}(km,k'm]  \right. \\
& + \left. N_2^{2H}(km,(k'-1)m] - N_2^{2H}((k-1)m,(k'-1)m] \right\} \\
 & = \sigma_1^2 C_{k,k'} + \sigma_2^2\theta^2 C'_{k,k'} \; .
\end{align*}
Hereafter, we only deal with the terms related to $N_1$, the other terms being
similarly dealt with and we omit the subscript~1.  For any $a,c\geq0$ and $b>0$,
since $0<2H<1$, we have
\begin{align*}
  0 & \leq  (a+b)^{2H} - b^{2H} - (a+b+c)^{2H} + (b+c)^{2H} \\
  & = 2H (1-2H) \int_b^{a+b} \int_z^{z+c} u^{2H-2} \mathrm{d} u \, \mathrm{d} z
  \leq ac b^{2H-2} \; .
\end{align*}
Applying this bound with $a_k=N((k-1)m,km]$, $b_{k,k'}=N(km,(k'-1)m]$ 
yields
\begin{align*}
 C_{k,k'}^2  \leq  b_{k,k'}^{4H-4} a_k^2 a_{k'}^2  \mathbf1_{\{b_{k,k'}>0\}} \; .
\end{align*}
Taking expectation, we have, by stationarity,
\begin{align*}
  m^{2-4H} n^{-2} &  \sum_{1 \leq k < k" \leq n/m} \esp[C_{k,k'}^2 \mathbf1_{\{b_{k,k'}>0\}} ] \\
  & \leq m^{1-4H} n^{-1} \sum_{k=1}^{n/m} \esp \left[ N^{4H-4}(km)
    \mathbf1_{\{N(km)>0\}} N^2(m) N^2((k-1)m,km] \right] \; .
\end{align*}
Applying H\"older's inequality yields
\begin{align*}
  m^{2-4H} n^{-2} &  \sum_{1 \leq k < k" \leq n/m} \esp[C_{k,k'}^2 \mathbf1_{\{b_{k,k'}>0\}} ]\\
  & \leq m^{-2} \esp^{1/2}[N^4(m)] \frac{m}n \sum_{k=1}^{n/m} \esp^{1/2} \left[
    \{N(km)/km\}^{8H-8} \mathbf1_{\{N(km)>0\}}\right] k^{4H-4} \; .
\end{align*}
By stationarity, $\esp[N^4(m)] = O(m^4)$, and applying Assumption~(\ref{eq:moment-cond-8}), we
obtain, for some constant $c$,
\begin{align*}
  m^{2-4H} n^{-2} & \sum_{1 \leq k < k" \leq n/m} \esp[C_{k,k'}^2 \mathbf1_{\{b_{k,k'}>0\}} ] \leq c \; \frac{m}n
  \sum_{k=1}^{n/m}  k^{4H-4}  = o(1) \; .
\end{align*}
Consider now the event $\{b_{k,k'}=0\}$. Then, using the above notations, we have
\begin{align*}
  \esp[ C_{k,k'}^2 \mathbf1_{\{b_{k,k'}=0\}}] & = \esp[\{a_k^{2H}+a_{k'}^{2H} -
  (a_k+a_{k'})^{2H}\}^2 \mathbf1_{\{b_{k,k'}=0\}}] \leq 4
  \esp[(a_k^{4H}+a_{k'}^{4H}) \mathbf1_{\{b_{k,k'}=0\}}] \; .
\end{align*}
Thus, by stationarity of $N$, we have 
\begin{align*}
  m^{2-4H} n^{-2} \sum_{1 \leq k < k" \leq n/m} & \esp[C_{k,k'}^2
  \mathbf1_{\{b_{k,k'}=0\}} ] \\
  & \leq m^{1-4H} n^{-1}  \sum_{k=1}^{n/m} \esp[N^{4H}(m) \mathbf 1_{\{N(m,km]=0\}}] \\
  & + m^{1-4H} n^{-1} \sum_{k=1}^{n/m} \esp[N^{4H}((k-1)m,km] \mathbf 1_{\{N((k-1)m)=0\}}] \; .
\end{align*}
Applying H\"older's inequality, $\esp[N^2(m)]=O(m^2)$ and stationarity  yields
\begin{align*}
  m^{1-4H} n^{-1} & \sum_{k=1}^{n/m} \esp[N^{4H}(m) \mathbf 1_{\{N(m,km]=0\}}] =
  m^{1-4H} n^{-1} \esp[N^{4H}(m) \mathbf1_{\{N(m,n]=0\}}]  \\
  & \leq m^{1-4H} n^{-1} \esp^{2H}[N^2(m)] \pr^{1-2H}(N(m,n]=0)  \leq C m n^{-1}  \pr^{1-2H}(N(n-m)=0) \; .
\end{align*}
Since $n/m\to\infty$, it holds that $\lim_{n\to\infty}\pr(N(n-m)=0) = 0$. Similarly, 
\begin{align*}
  m^{1-4H} n^{-1} & \sum_{k=1}^{n/m} \esp[N^{4H}((k-1)m,km] \mathbf 1_{\{N((k-1)m)=0\}}] \\
  & \leq m^{1-4H} n^{-1} \esp^{2H}[N^2(m)] \sum_{k=1}^{n/m} \pr^{1-2H}
  (N((k-1)m)=0) \leq \pr^{1-2H}(N(n)=0)
\end{align*}
and this last term tends to 0 as $n$ tends to infinity.  This concludes the
proof of~(\ref{eq:conv-var-con}) and of~\eqref{eq:SmConsistency}.
\end{proof}

\begin{proof}[Proof of Theorem~\ref{theo:ols-strong}]
  The proof is a consequence of the convergence~(\ref{eq:denominator}), the
  decomposition~(\ref{eq:cointegrating-error}), and
  Lemmas~\ref{lem:conv-loi-strong} and~\ref{lem:reste-strong}, whose assumptions
  are those of the Theorem.
\end{proof}

\begin{lemma}
  \label{lem:conv-loi-strong}
  Under the assumptions of Theorem~\ref{theo:ols-strong},
  \begin{gather}
  \label{eq:convergence-BdZ}
  n^{-H-1/2} \sum_{j=1}^n \{y_1^\eta(j)-\theta y_2^\eta(j)\} y_2(j)
  \dlim \Sigma_0 \int_0^1 B(s) \, \mathrm d B_H(s) \; .
\end{gather}
where $B_H$ is a standard fractional Brownian motion independent of $B$ and
\begin{align*}
  \Sigma_0 = (\theta^{-2} \lambda_{1}\sigma_{1,e}^2+\lambda_2 \sigma_{2,e}^2)
  (c_1^2 + \theta^2 c_2^2) \; .
\end{align*}
\end{lemma}

\begin{proof}  [Proof of Lemma~\ref{lem:conv-loi-strong}]
  Denote $S_n = \sum_{j=1}^n \{y_1^\eta(j)-\theta y_2^\eta(j)\} y_2(j)$ and
  write $y_2 = y_2^{e}+y_2^\eta$ with obvious notation. Denote $\zeta_j =
  y_1^\eta(j) - \theta y_2^\eta(j) = \xi_{1,N_1(j)} - \theta
  \xi_{2,N_2(j)}$. Then
  \begin{align}
    S_n = \sum_{j=1}^n \zeta_j y_2^{e}(j) & + \sum_{j=1}^n \zeta_j
    \xi_{2,N_2(j)} \; .  \label{eq:main}
  \end{align}
  By the last part of Assumption~\ref{hypo:microstructure-strong}, the last term
  in the righthand side of~(\ref{eq:main}) is $O_P(n)$.
Consider the first term in the righthand side of~(\ref{eq:main}), say
$S_{1,n}$. Write
  \begin{align*}
    S_{1,n} & = \sum_{j=1}^n \zeta_j \sum_{k=1}^{N_2(j)} e_{2,k} + \theta^{-1}
    \sum_{j=1}^n \zeta_j \sum_{k=1}^{N_1(t_{2,N_2(j)})} e_{1,k}
    \\
    & = \sum_{k=1}^{N_2(n)} e_{2,k} \sum_{ \{j \leq n :\, N_2(j) \geq k\}}
    \zeta_j + \theta^{-1} \sum_{k=1}^{N_1(t_{2,N_2(n)})} e_{1,k}
    \sum_{\{j\leq n: \, N_1(t_{2,N_2(j)}) \geq k\} } \zeta_j    \\
    & = T_{1,n} + \theta^{-1} T_{2,n} \; .
  \end{align*}
  Denote $W_n(t) = \sum_{j=1}^{[nt]} \zeta_j$. Since $N_2(j)<k$ iff
  $j<t_{2,k}$, we obtain
\begin{align*}
  T_{1,n} = y_2^{e_2}(n) W_n(1) - \sum_{k=1}^{N_2(n)} e_{2,k}
  W_n(t_{2,k}/n) \; .
\end{align*}
By Assumption~\ref{hypo:microstructure-strong} and
Theorem~\ref{theo:fclt-levels}, $n^{-1/2-H} y_2^{e_2}(n) W_n(1) \dlim
\sqrt{\lambda_2} \sigma_2 B_2(1) Z(1)$ with $Z = c_1B_H^{(1)} - \theta c_2
B_H^{(2)} \equallaw \sqrt{c_1^2+\theta^2c_2^2} \, B_H$. Let the last term be denoted by $U_n$. Since
the shocks $e_{i,k}$ are i.i.d. Gaussian, we can compute the characteristic
function of~$U_n$.
\begin{align*}
  \esp[\exp\{\mathrm i t n^{-1/2-H}  U_n\} ] & = \esp \left[ \exp
    \left\{ - \frac {\sigma_{2,e}^2 t^2} 2 \frac1n \sum_{k=1}^{N_2(n)} \left(
        n^{-H}  W_n(t_{2,k}/n) \right) ^2 \right\} \right]  \\
  & \to \esp \left[ \exp \left\{ - \frac {\lambda_2\sigma_{2,e}^2 t^2} 2 \int_0^1
      Z^2( s) \, \mathrm d s \right\} \right] \; .
\end{align*}
The convergence is actually joint with that of $n^{-1/2-H} y_2^{e_2}(n) W_n$,
thus we have
\begin{align*}
  n^{-1/2-H}  T_{1,n} \dlim \sqrt{\lambda_2} \sigma_{2,e} B_2(1)
  Z(1) - \sqrt{\lambda_2} \sigma_{2,e} \int_0^1 Z(s) \, \mathrm d B_2(s) \; .
\end{align*}
The limit can also be written as $\sqrt{\lambda_2}
\sigma_{2,e} \int_0^1 B_2(s) \, \mathrm d Z(s)$. Consider now the term
$T_{2,n}$. Note that $N_1(t_{2,N_2(j)})<k$ iff $j \leq t_{2,N_2(t_{1,k})+1}$. Thus
\begin{align*}
  T_{2,n} & = \sum_{k=1}^{N_1(t_{2,N_2(n)})} e_{1,k} W_n(1) -
  \sum_{k=1}^{N_1(t_{2,N_2(n)})} e_{1,k} W_n(t_{2,N_2(t_{1,k})+1}/n) \; .
\end{align*}
By similar arguments as previously, we obtain
\begin{align*}
  n^{-H-1/2} T_{2,n} \dlim \sqrt{\lambda_{1}} \sigma_1 B_1(1) Z(1) -
  \sqrt{\lambda_{1}} \sigma_1 \int_0^1 Z(s) \, \mathrm d B_1(s) \; .
\end{align*}
All convergences hold jointly, thus~(\ref{eq:convergence-BdZ}) holds.
\end{proof}

\begin{lemma}
  \label{lem:reste-strong}
Under the assumptions of Theorem~\ref{theo:ols-strong},
\begin{gather}
  \sum_{j=1}^n \{ r_{1,j} - \theta r_{2,j}\} y_2(j)  = O_P(n) \; ,
\end{gather}
\end{lemma}
\begin{proof}[Proof of Lemma~\ref{lem:reste-strong}]
We first study the term with $r_{1,j}$ and split it into three parts.
  \begin{align*}
    \sum_{j=1}^n r_{1,j} y_2(j) = \sum_{j=1}^n r_{1,j} y_2^{e_1}(j) + \sum_{j=1}^n
    r_{1,j} y_2^{e_2} (j) + \sum_{j=1}^n r_{1,j} y_2^\eta(j)
  \end{align*}
  We start with the last one. Recall that $N_1(t_{2,N_2(j)}) < k \leq N_1(j)$
  iff $t_{1,k} \leq j \leq t_{1,k} + A_2(t_{1,k})$.  Thus
\begin{align}
  \sum_{j=1}^n r_{1,j} y_2^\eta(j) & = \sum_{j=1}^n \xi_{2,N_2(j)}
  \sum_{N_1(t_{2,N_2(j)}) < k \leq N_1(j)} e_{1,k} \; . \label{eq:decomp-r1}
\end{align}
If the microstructure shocks are independent of the counting processes, then
\begin{multline*}
  \esp \left[ \left(\sum_{j=1}^n e_{1,k} \sum_{t_{1,k} \leq j < t_{1,k} +
        A_2(t_1,k)} \xi_{2,N_2(j)} \right)^2 \mid \mathcal N \right] \\
  = \sigma_{1,e}^2 \sum_{k=1}^{N_1(n)} \esp \left[ \left( \sum_{t_{1,k} \leq j
        < t_{1,k} + A_2(t_1,k)} \xi_{2,N_2(j)} \right)^2 \mid \mathcal N
  \right] \leq C \sum_{k=1}^{N_1(n)} (A_2(t_{1,k})+1)^2 \sup_\ell \esp[\xi_{2,\ell}^2] \; .
\end{multline*}
Conditioning on $N_1$ and then taking expectation yields
\begin{align*}
  \esp \Big [ \Big (\sum_{j=1}^n e_{1,k} \sum_{t_{1,k} \leq j < t_{1,k} +
    A_2(t_1,k)} \xi_{2,N_2(j)} \Big )^2 \Big ] & \leq C \esp[N_1(n)] \sup_t
  \esp[\{1+A_2(t)\}^2] \sup_\ell \esp[\xi_\ell^2] = O(n) \; .
\end{align*}
Consider now $R_{2,n} = \sum_{j=1}^n r_{1,j} y_2^{e_2}(j)$.
\begin{align*}
  R_{2,n} & = \sum_{j=1}^n y_2^{e_2}(j) \sum_{N_1(t_{2,N_2(j)})+1}^{N_1(j)}
  e_{1,k} = \sum_{k=1}^{N_1(n)} e_{1,k} \sum_{t_{1,k} \leq j <
    t_{1,k}+A_2(t_{1,k})} y_2^{e_2}(j) \; .
\end{align*}
By independence of the efficient shocks and the counting processes, we have
\begin{align*}
  \esp[R_{2,n}^2 \mid \mathcal N] \leq C N_1(n) \sum_{k=1}^{N_1(n)}
  (A_2(t_{1,k})+1)^2 = O_P(n^2) \; .
\end{align*}
This proves that $R_{2,n} = O_P(n)$. Consider finally $R_{1,n} = \sum_{j=1}^n
r_{1,j} y_2^{e_1}(j)$. By definition, $e_{1,k}$ is independent of
$y_2^{e_1}(j)$ for $j$ such that $N_1(t_{2,N_2(j)}) <k$. Thus, we can compute
the conditional variance given $\mathcal N$.
\begin{align*}
  \esp[ R_{1,n}^2 \mid \mathcal N] & = \sigma_{1,e}^2 \sum_{k=1}^{N_1(n)}
  \esp\left[ \left( \sum_{t_{1,k} \leq j < t_{1,k}+A_2(t_{1,k})} y_2^{e_1}(j)
    \right)^2 \mid \mathcal N \right] \\
  & \leq C N_2(n) \sum_{k=1}^{N_1(n)} (A_2(t_{1,k})+1)^2 = O_P(n)
\end{align*}
by~\eqref{eq:pp-renforce-2}.  This concludes the proof of
Lemma~\ref{lem:reste-strong}.
\end{proof}

\subsection{Proof of Theorems~\ref{theo:taper-weak} and~\ref{theo:taper-strong}}
\label{subsec:prooftaper}
Write
\begin{align*}
  \tilde \theta_n = \theta + \frac{\sum_{\ell=1}^m d_{\Delta r,\ell} \, \bar
    d_{\Delta y_2,\ell}}{\sum_{\ell=1}^m |d_{\Delta y_2,\ell}|^2} +
  \frac{\sum_{\ell=1}^m d_{\Delta y^\eta,\ell} \, \bar d_{\Delta
      y_2,\ell}}{\sum_{\ell=1}^m |d_{\Delta y_2,\ell}|^2} \;
\end{align*}
with $y^\eta(j) = y_1^\eta(j) - \theta y_2^\eta(j)$, $r(j) = r_1(j) -\theta
r_2(j)$ and
\begin{align*}
  r_1(j) = \sum_{k=N_1(t_{2,N_2(j)})+1}^{N_1(j)} e_{1,k} \; , \ \ \ r_{2}(j) =
  \sum_{k=N_2(t_{1,N_1(j)})+1}^{N_2(j)} e_{2,k} \; .
\end{align*}
By summation by parts, since $h(0)=h(1)=0$, for any time series $\{x_j\}$, we
can write
\begin{align}
  \label{eq:summation-by-parts}
  d_{\Delta x,\ell} = \sum_{j=0}^{n-1} \{h_\ell(j/n) - h_\ell((j+1)/n) \} x_j =
  -\frac 1n \sum_{j=0}^{n-1} w_\ell(j,n) x_j
\end{align}
with  $w_\ell(j,n) = n\{h_\ell((j+1)/n) - h_\ell(j/n)\}$.
Applying~(\ref{eq:summation-by-parts}) to $y_2$ yields
\begin{align*}
  d_{\Delta y_2,\ell} = - \frac1n \sum_{j=0}^{n-1} w_\ell(j,n) y_2(j) \; .
\end{align*}
Since the assumptions of Theorems~\ref{theo:taper-weak}
and~\ref{theo:taper-strong} imply those of Theorem~\ref{theo:fclt-levels}, the
Continuous Mapping Theorem yields
\begin{align}
  \label{eq:continuous-mapping-taper}
  \{ n^{-1/2} d_{\Delta y_2,\ell}, 1 \leq \ell \leq m\} \dlim \left\{- \Sigma_e
    \int_0^1 h_\ell'(s) B(s) \, \mathrm d s \; , 1 \leq \ell \leq m \right\}
\end{align}
where $B$ is a standard Brownian motion and $\Sigma_e^2 = \theta^{-2}
\lambda_{1}\sigma_{1,e}^2+\lambda_2 \sigma_{2,e}^2$.  By integration by parts,
the integral can also be expressed as
\begin{align*}
  - \int_0^1 h_\ell'(s) B(s) \, \mathrm d s = \int_0^1 h_\ell(s) \, \mathrm d
  B(s) \; .
\end{align*}
This in turn implies
\begin{align}
  \label{eq:denominator-taper}
  n^{-1} \sum_{\ell=1}^m |d_{\Delta y_2,\ell}|^2 \dlim \Sigma_e^2 \sum_{\ell=1}^m
  \left| \int_0^1 h_\ell(s) \, \mathrm d B(s) \right|^2 \; .
\end{align}
Applying now~(\ref{eq:summation-by-parts}) to $y^\eta$ we obtain
\begin{align*}
  d_{\Delta y^\eta,\ell} = - \frac 1n \sum_{j=0}^{n-1} w_\ell(j,n) \{
  y_1^\eta(j)-\theta y_2^\eta(j)\} \; .
\end{align*}
In the case of weak fractional cointegration, we apply
Assumption~\ref{hypo:microstructure-weak}, the Continuous Mapping Theorem and
integration by parts to obtain
\begin{align}
  \label{eq:numerator-taper-weak}
  n^{-H} \ell(n) d_{\Delta y^\eta,\ell} = -n^{-1-H} \ell(n) \sum_{j=0}^{n-1}
  w_\ell(j,n) \{y_1^\eta(j)-\theta y_2^\eta(j)\} \dlim \int_0^1 h_\ell(t) \,
  \mathrm d Z_H(t) \;
\end{align}
where, by independence of $B_H^{(1)}$ and $B_H^{(2)}$,
$$
Z_H(t) = c_1 B_H^{(1)}(\lambda_1t) - \theta c_2 B_H^{(2)}(\lambda_2t) \equallaw
\sqrt{\lambda_1^{2H} c_1^2 + \lambda_2^{2H} \theta^2 c_2^2 } \, B_H
$$
and $B_H$ is a standard fractional Brownian motion.  The first part of
Lemma~\ref{lem:reste-taper} shows that $d_{\Delta r,\ell}$ is negligible under
the assumptions of Theorem~\ref{theo:taper-weak}.  This, and the
convergences~(\ref{eq:continuous-mapping-taper}), (\ref{eq:denominator-taper})
and~(\ref{eq:numerator-taper-weak}) conclude the proof of
Theorem~\ref{theo:taper-weak}. \hfill \qed

We now prove Theorem~\ref{theo:taper-strong}. Since
$h_\ell(0)=h_\ell(1)=0$, we have $\sum_{j=0}^{n-1} w_{\ell}(j,n) = 0$,
hence
\begin{align*}
  \sum_{j=0}^{n-1} w_\ell(j,n) y_i^\eta(j) = \sum_{j=0}^{n-1} w_\ell(j,n)
 ( \xi_{i,N_i(j)}-\xi_{i,0}) = \sum_{j=0}^{n-1} w_\ell(j,n)  (\xi_{i,N_i(j)}-\mu_i^*) \; .
\end{align*}
Denote $S_{i,0}=0$ and for $k\geq1$, $S_{i,k} = \sum_{j=1}^k
(\xi_{i,N_i(j)}-\mu_i^*)$.  Define $\omega_\ell(j,n) = n \{w_\ell(j+1,n) -
w_\ell(j,n)\}$.  Applying again summation by parts, we have
\begin{align*}
  \sum_{j=0}^{n-1} w_\ell(j,n) y_i^\eta(j) & = - \frac1n \sum_{j=1}^{n-1}
  \omega_\ell(j,n) S_{i,j} + w_\ell(n,n) S_{i,n-1} + w_\ell(0,n) (\xi_{i,0}-\mu_i^*) \; ,
\end{align*}
Under Assumption~\ref{hypo:microstructure-strong}, by the
Continuous Mapping Theorem, we obtain
\begin{align}
  n^{1-\gamma} \ell(n) d_{\Delta y^\eta,\ell} & = - n^{-\gamma} \ell(n)
  \sum_{j=1}^n w_\ell(j,n) y_i^\eta(j) \nonumber \\
&  \dlim \int_0^1 h_\ell^{\prime\prime}(t) B_H^{(i)}(t) \, \mathrm d t - h'(1)
  B_H^{(i)}(1) \equallaw - \int_0^1 h_\ell'(s) \, \mathrm d B_H^{(i)}(s) \; .
  \label{eq:numerator-taper-strong}
\end{align}
The second part of Lemma~\ref{lem:reste-taper} implies that the term $d_{\Delta
  r,\ell}$ does not contribute to the limit under the Assumptions of
Theorem~\ref{theo:taper-strong}. This, and the
convergences~(\ref{eq:continuous-mapping-taper}), (\ref{eq:denominator-taper})
and~(\ref{eq:numerator-taper-strong}) conclude the proof of
Theorem~\ref{theo:taper-strong}. \hfill \qed

\begin{lemma}
 \label{lem:reste-taper}
 Under the assumptions of Theorem~\ref{theo:taper-weak}, then $d_{\Delta r,\ell}
 = O_P(1)$. Under the assumptions of Theorem~\ref{theo:taper-strong}, then
 $d_{\Delta r,\ell} = O_P(n^{-1/2})$.
\end{lemma}
\begin{proof}
  Applying~(\ref{eq:summation-by-parts}) to $r$, we see that we only need to
  prove that the independence between the counting processes and the efficient
  shocks and~\eqref{eq:pp-renforce} implies that $\sum_{j=1}^n w_\ell(j,n)
  r_{i,j} = O_p(n)$ and~\eqref{eq:pp-renforce-2} implies that $\sum_{j=1}^n
  w_\ell(j,n) r_{i,j} = O_p(n^{1/2})$. We start with $r_1$.
  \begin{align*}
    \sum_{j=1}^n w_\ell(j,n) r_{1,j} = \sum_{k=1}^{N_1(n)} e_{1,k}
    \sum_{t_{1,k} \leq j < t_{1,k}+A_2(t_{1,k})} w_\ell(j,n) \; .
  \end{align*}
 Taking conditional expectation yields, for $q=1,2$,
  \begin{align*}
    \esp\left[ \left| \sum_{j=1}^n w_\ell(j,n) r_{1,j} \right|^q \mid
      \mathcal N \right]  \leq C \sum_{k=1}^{N_1(n)} (A_2(t_{1,k})+1)^q  \; .
  \end{align*}
  Applying~\eqref{eq:pp-renforce} if $q=1$ and~\eqref{eq:pp-renforce-2} if $q=2$
  shows that the last term is $O_P(n)$.  This proves that $\sum_{j=1}^n
  w_\ell(j,n) r_{1,j} = O_P(n)$ under the assumptions of Theorem~\ref{theo:taper-weak}
  and $O_P(\sqrt n)$ under the assumptions of Theorem~\ref{theo:taper-strong}. The term
  $\sum_{j=1}^n w_\ell(j,n) r_{2,j}$ is dealt with similarly.
\end{proof}

\subsection{Proof of Theorem~\ref{theo:strong-continuous-taper}}
Write
\begin{align*}
  \tilde \theta_{n,\delta} = \theta + \frac{\sum_{\ell=1}^m d_{\Delta \tilde r,\ell}
    \bar d_{\Delta \tilde y_2,\ell}}{\sum_{\ell=1}^m |d_{\Delta \tilde
      y_2,\ell}|^2} + \frac{\sum_{\ell=1}^m d_{\Delta \tilde y^\eta,\ell} \bar
    d_{\Delta \tilde y_2,\ell}}{\sum_{\ell=1}^m |d_{\Delta \tilde y_2,\ell}|^2}
\end{align*}
with $\tilde y(j) = \tilde y_1^\eta(j) - \theta \tilde y_2^\eta(j)$, $\tilde
r(j) = \tilde r_1(j) - \theta \tilde r_2(j)$ and
\begin{align*}
  r_1(s) & = \sum_{k=N_1(t_{2,N_2(s)})+1}^{N_1(s)} e_{1,k} \; , \ \ \ r_2(s) =
  \sum_{k=N_2(t_{1,N_1(s)})+1}^{N_2(s)} e_{2,k} \; .
\end{align*}
and the DFT is defined as in~(\ref{eq:taperedDFT}). Applying summation by parts as
in~(\ref{eq:summation-by-parts}), we obtain
\begin{align*}
  d_{\Delta \tilde y_2,\ell} = - \frac1n \sum_{j=0}^{n-1} w_\ell(j,n) \tilde
  y_2(j) = - \frac1n \int_0^{n\delta} w_\ell(\lceil s/\delta\rceil,n) y_2(s) \,
  \mathrm d s = - \int_0^{\delta} w_\ell(\lceil nt/\delta\rceil,n) y_2(ns) \,
  \mathrm d t \; ,
\end{align*}
with $w_\ell(j,n) = n\{h_\ell((j+1)/n) - h_\ell(j/n)\}$ as before, and $\lceil t
\rceil$ is the smallest integer larger than or equal to $t$.  This yields
\begin{align*}
  n^{-1/2} d_{\tilde y_2,\ell} & \dlim - \Sigma_e \int_0^1 h_\ell'(s)
  B(\delta s) \, \mathrm d s \equallaw \Sigma_e \int_0^1 h_\ell(s) \,
  \mathrm d B(s) \; .
\end{align*}
Since $\eta_j = \xi_j - \xi_{j-1}$, we have
\begin{align*}
  \tilde y_i^\eta(j) = \int_{(j-1)\delta}^{j\delta} \xi_{i,N_i(s)} \, \mathrm d
  s - \delta \xi_{i,0} \; .
\end{align*}
Differencing cancels the term
$\delta\xi_0$. Applying~(\ref{eq:summation-by-parts}) and summation by parts and
the property that $\sum_{j=0}^{n-1}w_\ell(j,n)=0$, we obtain
\begin{align*}
  d_{\Delta\tilde y_i^\eta,\ell} & = - \frac1n \sum_{j=0}^{n-1} w_\ell(j,n)
  \int_{(j-1)\delta}^{j\delta} \xi_{i,N_i(s)} \, \mathrm d s = - \frac1n
  \sum_{j=0}^{n-1} w_\ell(j,n) \int_{(j-1)\delta}^{j\delta} \{ \xi_{i,N_i(s)} - \mu_i^*\}
  \, \mathrm d s   \\
  & = \frac 1 {n^2} \sum_{j=1}^{n-1} \omega_\ell(j,n) \int_{0}^{j\delta}
   \{ \xi_{i,N_i(s)} - \mu_i^*\} \, \mathrm d s - \frac1n w_\ell(n,n) \int_0^{(n-1)\delta}
   \{ \xi_{i,N_i(s)} - \mu_i^*\} \, \mathrm d s \; .
\end{align*}
Under Assumption~\ref{hypo:microstructure-strong-taper-continuous}, we thus
have, with $Z= B_H^{(1)}-\theta B_H^{(2)}$,
\begin{align*}
  n^{1-H} \{d_{\Delta \tilde y_1^\eta, \ell} - \theta d_{\Delta \tilde
    y_2^\eta,\ell} \} \dlim \int_0^1 h_\ell''(s) Z(\delta s) \, \mathrm d s -
  h'(1) Z(\delta) \; .
\end{align*}
We must now deal with the remaining terms of the cointegrating error. If
$H>1/2$, Lemma~\ref{lem:strong-standard} implies that the term $d_{\Delta\tilde
  r,\ell}$ does not contribute to the limit. If $H=1/2$, both terms are of the
same order.  This concludes the proof of
Theorem~\ref{hypo:microstructure-strong-taper-continuous}. \hfill \qed

\begin{lemma}
   \label{lem:strong-standard}
   Under the assumptions of Theorem~\ref{theo:strong-continuous-taper}
  \begin{align*}
    d_{\Delta \tilde r_i,\ell} = O_P(n^{-1/2}) \; .
  \end{align*}
\end{lemma}
\begin{proof}
 Applying as usual summation by parts, we obtain
\begin{align}
  d_{\Delta\tilde r_1,\ell} & = - \frac 1n \sum_{k=1}^{N_1(n\delta)} e_{1,k}
  \sum_{j=1}^n w_\ell(j,n) \int_{(j-1)\delta}^{j\delta} \mathbf1_{\{t_{1,k}
    \leq s < t_{1,k} + A_2(t_{1,k})\}} \, \mathrm d s \nonumber
  \\
  & = -\frac 1n \sum_{k=1}^{N_1(n\delta)} e_{1,k} \int_{0}^{n\delta} w_\ell
  (\lceil s/\delta \rceil, n) \mathbf1_{\{t_{1,k} \leq s < t_{1,k} +
    A_2(t_{1,k})\}} \, \mathrm d s \nonumber
  \\
  & = -\frac 1n \sum_{k=1}^{N_1(n\delta)} e_{1,k} \int_{t_{1,k} \wedge (n\delta)}^{\{t_{1,k} +
    A_2(t_{1,k})\} \wedge (n\delta)} w_\ell (\lceil s/\delta \rceil, n) \,
  \mathrm d s   \nonumber
  \\
  & = - \sum_{k=1}^{N_1(n\delta)} e_{1,k} \int_{(t_{1,k}/n)\wedge \delta}^{\{(t_{1,k} +
    A_2(t_{1,k}))/n\} \wedge \delta} w_\ell (\lceil nt/\delta \rceil, n) \,
  \mathrm d t \; . \label{eq:the-expression}
  \end{align}
  Taking conditional expectation and applying~(\ref{eq:pp-renforce-2}), we obtain
\begin{align*}
  \esp \left[ |d_{\Delta \tilde r_1,\ell}|^2 \mid \mathcal N \right] \leq \frac
  C {n^2} \sum_{k=1}^{N(n)} A_2^2(t_{1,k}) = O_P(n^{-1}) \; .
\end{align*}

\end{proof}

\subsection{Additional Lemmas}
\label{sec:additional}

\begin{lemma}
  \label{lem:moment-frt}
  If the durations $t_{i,k}-t_{i,k-1}$ form a stationary ergodic sequence with
  finite moment of order $2p+1$, if $\pr(t_{i,1}>0)=1$ and if the associated
  point process has finite intensity, then
  \begin{align*}
    \sup_{s\geq0} \esp[(t_{i,N_i(s)+1}-s)^p] < \infty \; .
  \end{align*}
\end{lemma}
\begin{proof}[Proof of Lemma~\ref{lem:moment-frt}]
  We omit the index $i$. Let $\theta_t$ denote the shift operator and let $A(t)$
  be the forward recurrence time. Then $A(s) = t_{N(s)+1}-s = t_1 \circ
  \theta_s$.  Since the sequence $\{\tau_i\}$ is stationary under $\pr$,
  there exists a probability law $P^*$ such that $N$ is a stationary ergodic
  point process under $P^*$, see \citet[Section
  1.3.5]{baccelli:bremaud:2003}. Applying \citet[Formula
  1.3.3]{baccelli:bremaud:2003}, we obtain
  \begin{align}
    \esp[A^p(s)] & = \lambda^{-1} \esp^* \left[ \sum_{k=1}^{N(1)} t_1^p \circ
      \theta_s \circ \theta_{t_k} \right] = \lambda^{-1} \esp^* \left[
      \sum_{k=1}^{N(1)} A^p(s+t_k) \right] \nonumber    \\
    & = \lambda^{-1} \esp^* \left[ \sum_{k=1}^{N(1)} \{t_{N(s+t_k)+1} - s-t_k\}^p
    \right] \leq \lambda^{-1} \esp^* \left[ \sum_{k=1}^{N(1)} \{t_{N(s+1)+1} -
      s\}^p \right] \nonumber    \\
    & = \lambda^{-1} \esp^* [ N(1) \{t_{N(s+1)+1} - s\} ^p] \leq \lambda^{-1} \{
    \esp^* [ N(1)^2] \}^{1/2} \{\esp^*[(t_{N(s+1)+1} - s)^{2p} ]\}^{1/2} \;
    . \label{eq:quick-dirty}
  \end{align}
  Since $N$ is stationary under $P^*$, the last term does not depend on $s$, and
  by the Ryll-Nardzewski inversion formula \cite[Formula
  1.2.25]{baccelli:bremaud:2003}, we have
  \begin{align*}
    \esp^*[(t_{N(s+1)+1} - s)^{2p} ] = \esp^*[(t_1+1)^{2p}] = \lambda
    \esp[\int_0^{t_1} (t_1+1-s)^{2p}\, \mathrm d s \leq \lambda \esp[(1+t_1)^{2p+1}]
  \end{align*}
  By \citet[Property 1.6.3]{baccelli:bremaud:2003}, the point process $N$ is
  stationary and ergodic under $P^*$ since the sequence of durations $\tau_k$ is
  stationary and ergodic. Thus, by \citet[Theorem
  3.5.III]{daley:vere-jones:2003}, $\esp^*[N(0,1)^2]<\infty$.  Plugging the last
  two bounds into~(\ref{eq:quick-dirty}), we obtain that $\esp[A^p(s)]$ is
  uniformly bounded.
\end{proof}

\begin{lemma}
  \label{lem:time-deformation}
  Assume that there exists an increasing sequence $\{s_n, n\geq0\}$ such that $s_0=0$ and
  \begin{enumerate}[(a)]
  \item $f$ is either constant or strictly increasing and differentiable on
    $(s_n,s_{n+1})$ and the jumps of $f$ occur at some (but not necessarily all)
    of the $s_n$;
  \item if $f$ is eiter constant or increasing on both intervals $(s_n,s_{n+1})$
    and $(s_{n+1},s_{n+2})$, then $f$ has a jump at $s_{n+1}$.
  \end{enumerate}
  Assume moreover that
\begin{itemize}
\item \emph{(minimum duration of trading and nontrading periods)} there exists
  $\delta_0>0$ such that $s_{n+1}-s_n\geq \delta_0$ for all $n\geq0$;
\item \emph{(maximum duration of nontrading periods)} there exists $C_0$ such
  that for all $n\geq0$, if $f$ is constant on $(s_n,s_{n+1})$, then
  $s_{n+1}-s_n \leq C_0$;
\item \emph{(non stoppage of time during trading periods)} there exists
  $\delta_1>0$ such that for all $n\geq0$, $f$ is either constant on
  $(s_n,s_{n+1})$, or $f'(t) \geq \delta_1 $ for all $t \in(s_n,s_{n+1})$.
\end{itemize}

Let $\tilde N$ be a point process with event times $\{\tilde t_k\}$ and let $N$
be the point process defined by $N(\cdot) = \tilde N(f(\cdot))$ with event times
$\{t_k\}$.  If $\sup_{s\geq0} \esp[(\tilde t_{\tilde N(s)+1}-s)^p] < \infty$,
then $\sup_{s\geq0} \esp[( t_{ N(s)+1}-s)^p] < \infty$.
\end{lemma}


\begin{proof}[Proof of Lemma~\ref{lem:time-deformation}]
Define the nondecreasing left-continuous inverse $f^\leftarrow$ of a nondecreasing c\`adl\`ag function $f$ by
\begin{align*}
  f^\leftarrow (u) = \inf\{t \mid f(t) \geq u\} \; .
\end{align*}
  Note first that $f^\leftarrow (u) \leq t$ if only if $u \leq f(t)$ and
  $f^\leftarrow (f(t)) \leq t$. Thus we see that
\begin{align*}
  f^\leftarrow(\tilde t_n) \leq t & \Leftrightarrow \tilde t_n \leq f(t) \\
  & \Leftrightarrow \tilde N(f(t)) \geq n \\
  & \Leftrightarrow N(t) \geq n \; .
\end{align*}
This characterizes the sequence $\{t_n\}$, thus we obtain that $t_n =
f^\leftarrow (\tilde t_n)$.  The assumptions on $f$ imply the following
properties of $f^\leftarrow$.
\begin{itemize}
\item The jumps of $f^\leftarrow$ correspond to the intervals $(s_n,s_{n+1})$
  where $f$ is constant. More precisely, if $f$ is constant on $(s_n,s_{n+1})$,
  then $f^\leftarrow$ has a jump at $f(s_n)$ of size $s_{n+1}-s_n$. Since
  $f^\leftarrow$ is left continuous, it holds that
\begin{align*}
  f^\leftarrow (f(s_n)) = s_n \; , \ \ \lim_{u \to f(s_n), u>f(s_n)} = s_{n+1} \; .
\end{align*}
  Thus the jumps of $f^\leftarrow$ are of size $C_0$ at most.
\item If $f$ is increasing on an interval $(s_n,s_{n+1})$, then $f^\leftarrow$
  is differentiable on $(f(s_n),f(s_n^-))$ and $(f^\leftarrow)'(t) \leq
  \delta_1^{-1}$ for all $t \in (f(s_n),f(s_n^-))$.
\item The jumps of $f$ create no singularity in $f^\leftarrow$. If
  $f(s_n)>f(s_n^-)$, then $f^\leftarrow$ is constant on the interval
  $(f(s_n^-),f(s_n))$.
\end{itemize}
Let $\lceil x \rceil$ denote the smallest integer greater than or equal to the
real number $x$. Then, for $0 \leq s \leq t$,
\begin{align*}
  0 \leq f^\leftarrow (t) - f^\leftarrow (s) \leq C_0 \Big\lceil \frac{t-s}{\delta_0}
  \Big\rceil + \delta_1^{-1} (t-s) \; .
\end{align*}
Thus, there exits constants $c_1,c_2$ such that for all $s\leq t$,
\begin{align*}
  0 \leq f(t) - f(s) \leq c_1 + c_2(t-s) \; .
\end{align*}
Consider now the forward recurrence time of the point process $N$. Then
\begin{align*}
  0 & \leq t_{N(s)+1}-s = f^\leftarrow (\tilde t_{\tilde N(s)+1}) - f^\leftarrow (f(s)) + f^\leftarrow (f(s)) -s \\
  & \leq f^\leftarrow (\tilde t_{\tilde N(f(s))+1}) - f^\leftarrow (f(s)) \leq
  c_1 + c_2 \{\tilde t_{\tilde N(f(s))+1} - f(s)\} \; .
\end{align*}
Thus, there exists constants $c_3$ and $c_4$ such that
\begin{align*}
  \sup_{s\geq0} \esp[(t_{N(s)+1}-s)^p] \leq c_3 + c_4 \sup_{s\geq0} \esp[(\tilde
  t_{\tilde N(s)+1}-s)^p]
\end{align*}

\end{proof}

\begin{lemma}
  \label{lem:check-lmsd}
  Let $\{\epsilon_{k}\}$ be a sequence of i.i.d. positive random variables with
  finite mean $\mu_\epsilon$. Let $\{Y_{k}\}$ be a stationary standard Gaussian
  process such that
  \begin{align}
    \label{eq:hurstY-H}
    \mathrm{cov}(Y_{0},Y_{k}) =  \ell(n) n^{2H-2}
  \end{align}
  for $H \in (1/2,1)$ and $\ell$ a slowly varying function. For $k\geq1$, define
\begin{align*}
  \tau_k = \epsilon_k \mathrm e^{\sigma Y_k} \; .
\end{align*}
Then the sequence $\{\tau_k\}$ is ergodic and Assumption~\ref{hypo:lfgn-dates}
holds with $\lambda^{-1} = \mu_\epsilon \mathrm e^{\sigma^2/2}$.  If
$\pr(\epsilon_1>0)=1$ the Assumption~\ref{hypo:lfgn-pp} holds with
$\mu=\lambda = \mu_\epsilon^{-1} \mathrm e^{-\sigma^2/2}$. If moreover
$\esp[\epsilon_1^q]<\infty$ for all $q\geq1$, then (\ref{eq:pp-renforce}) and
(\ref{eq:pp-renforce-2}) hold.
\end{lemma}

\begin{remark}
  If instead of~(\ref{eq:hurstY-H}) we assume that
  \begin{align*}
    \sum_{k=1}^\infty |\mathrm{cov}(Y_0,Y_k)| < \infty \; ,
  \end{align*}
  then the moment requirement can be relaxed to $\esp[\epsilon_1^3]<\infty$ to
  obtain~(\ref{eq:pp-renforce}) and $\esp[\epsilon_1^5]<\infty$ to
  obtain~(\ref{eq:pp-renforce-2}).
\end{remark}

\begin{proof}[Proof of Lemma~\ref{lem:check-lmsd}]
  Note first that $\esp[\tau_{k}^p]<\infty$ as long as
  $\esp[\epsilon_{1}^p]<\infty$.  By Lemma~\ref{lem:moment-frt}, in order to
  check condition~(\ref{eq:pp-renforce}), we must only prove that the induced
  point process has finite intensiy, i.e.  there exists $t>0$ such that
  $\esp[N(t)]<\infty$. See \citet[Section 1.3.5]{baccelli:bremaud:2003}. Note
  that
  \begin{align*}
    \esp[N(x)] = \sum_{k=1}^\infty \pr(N(x) \geq k) = \sum_{k=1}^\infty
    \pr(t_k \leq x) \; .
  \end{align*}
  Thus, it suffices to prove that the series on the righthand side is
  summable. Denote $\mu=\esp[\tau_k]$ and $\rho_n =
  \mathrm{cov}(Y_0,Y_n)$. Applying \citet [Proposition~1]
  {deo:hurvich:soulier:wang:2009}, we have
  \begin{align*}
    \esp \left[ \left| \sum_{k=1}^n \tau_{k}-n\mu \right|^p
    \right] = O(v_n^p) \;
  \end{align*}
  with $v_n = n^H \ell(n)$.  If $\esp[\epsilon_1^p]<\infty$ for $p$ such that
  $p(1-H)>1$, for $n$ such that $n\mu>x$, it holds that
  \begin{align*}
    \pr(t_k \leq x) = O(x^{-1} v_k^p)
  \end{align*}
  and this series is summable.
\end{proof}




\begin{lemma}
  \label{lem:strong-noleverage-continuous}
  Assume that $\{\tau_k\}$ and $\{\xi_{k}\}$ are mutually independent stationary
  sequences such that $\esp[\xi_k]=0$, $\esp[\tau_k^2] <\infty$ and
  $\esp[\xi_{k}^2] < \infty$. Assume that the sequence of durations is weakly
  stationary and that $\mathrm{cov}(\tau_0,\tau_n)=0(n^{-\delta})$ for some
  $\delta>0$ and $\sup_{s\geq0} \esp[t_{N(s)+1}-s]<\infty$. Assume that
  $\mathrm{cov}(\xi_1,\xi_n) \sim c n^{2H-2}$, with $H\in(1/2,1)$ and $c>0$, and
  that
  \begin{align*}
    n^{-H}  \sum_{k=1}^{[n\cdot]} \xi_k \func c' B_H
  \end{align*}
  for some $c'>0$.  Then
$$
n^{-H} \int_0^{Tt} \xi_{N(s)} \, \mathrm d s \func c^{\prime\prime}B_H(t) \;
$$
for some $c^{\prime\prime}>0$.
\end{lemma}

\begin{proof}[Proof of Lemma~\ref{lem:strong-noleverage-continuous}]
  Denote $\esp[\tau_k]=\mu>0$.
\begin{align*}
  \int_0^{T} \xi_{N(s)} \, \mathrm d s & = \sum_{k=0}^{N(T)} \tau_{k+1} \xi_k
  - (t_{N(T)+1}-T)  \xi_{N(T)+1}  \\
  & = \sum_{k=0}^{N(T)} (\tau_{k+1}-\mu) \xi_{k} + \mu \sum_{k=0}^{N(T)} \xi_k -
  (t_{N(T)+1}-T) \xi_{N(T)+1}\; .
  \end{align*}
  By independence of $\{\tau_k\}$ and $\{\xi_k\}$, we have (assuming without
  loss of generality that $2H-\delta>1$),
\begin{align*}
  \mathrm {var} \left( \sum_{k=0}^{n} (\tau_{k+1} - \mu) \xi_k \right) =
  O(n^{2H-\delta}) \; .
\end{align*}
Thus, $n^{-H} \sum_{k=0}^{[n\cdot]} (\tau_{k+1} -\mu) \xi_k \func 0$.
Hence by the continuous mapping theorem, it also holds that $n^{-H}
\sum_{k=0}^{N(T\cdot)} (\tau_{k+1} -\mu) \xi_k \func 0$. By independence
and by assumption, $(t_{N(t)+1}-T)\xi_{N(T)} = O_P(1)$. By the continuous mapping
theorem, $n^{-H}\sum_{k=0}^{N(Tt)} \xi_k \func c'B_H(\mu^{-1} t)$.
\end{proof}

\begin{lemma}
 \label{lem:standard-dependence-noleverage}
 Let $\{\tau_k\}$, $\{V_k\}$ and $\{\zeta_k\}$ be sequences of random variables
 such that
  \begin{itemize}
  \item $\{\zeta_k\}$ is an i.i.d. sequence of zero-mean and unit variance
    random variables; $\{\tau_k\}$ and $\{V_k\}$ are sequences of positive
    random variables;
  \item the sequences $\{(\tau_k,V_k)\}$ and $\{\zeta_k\}$ are mutually independent;
  \item there exists $s>0$ such that $n^{-1}\sum_{k=1}^n \tau_{k+1}^2 V_k^2 \plim s^2$;
  \item $\sup_{k\geq0} \esp[\tau_{k+1}^{2+\varepsilon} V_k^{2+\varepsilon}] <
    \infty$ for some $\varepsilon>0$;
  \item $\sup_{s\geq0}\esp[t_{N(s)+1}-s]<\infty$.
  \end{itemize}
  Define $\xi_k = \zeta_k V_k$. Then $ T^{-1/2}\int_0^{T\cdot} \xi_{N(s)} \, \mathrm
  d s \func cB$ for some $c>0$.
\end{lemma}

\begin{proof}
  Let $\mathcal F_k$ be the sigma-field generated by random variables
  $\{\tau_{j+1},\zeta_j,V_j, j\leq k\}$.  Then $\esp[ \xi_k \tau_{k+1} \mid
  \mathcal F_{k-1}] = \tau_{k+1} V_k \esp[\zeta_k] = 0$. Thus,
  $\{\tau_{k+1}\xi_k\}$ is a martingale difference sequence.
  Under the stated assumptions, the martingale invariance principle
  \citet[Theorem~4.1]{MR624435} yields that $n^{-1/2} \sum_{k=1}^{[n\cdot]}
  \tau_{k+1}  \xi_k \func cB$ for some $c>0$.  As in the proof of
  Lemma~\ref{lem:strong-noleverage-continuous}, denote $\esp[\tau_k]=\mu>0$ and
  write
\begin{align*}
  \int_0^{T} \xi_{N(s)} \, \mathrm d s & = \sum_{k=0}^{N(T)} \tau_{k+1} \xi_k
  + (t_{N(T)+1}-T)  \xi_{N(T)}  \; .
  \end{align*}
  By the continuous mapping theorem, we have that $T^{-1/2} \sum_{k=1}^{N(T\cdot
    )} \tau_{k} \xi_{-1} \func \lambda cB$. As previously, the last term is a
  negligible edge effect. This concludes the proof.
\end{proof}

\begin{lemma}
   \label{lem:negative-moments}  
   Let $N$ be a stationary point process under $P$ with intensity $\lambda$ and
   let $P^0$ denote the Palm probability associated to $P$. Let $\gamma>0$.
   Assume that there exist $\delta\in(0,1)$ and $q>0$ such that
   \begin{align}
   \label{eq:bound-tk-palm}
     \sup_{k\geq1} k^{-q\delta}\esp^0[|t_k-\lambda^{-1}k|^q] < \infty \; .
   \end{align}
   If (\ref{eq:bound-tk-palm}) holds with $q\geq\gamma+1$, then
  \begin{align}
  \label{eq:moment-cond}
  \sup_{t\geq 2} \esp\left[ \left(
      \frac{N_i(t)}t\right)^{-\gamma}\mathbf1_{\{N_i(t)>0\}} \right] < \infty \; .
  \end{align}
  If~(\ref{eq:bound-tk-palm}) holds with $q>1+\gamma/(1-\delta)$, then
  $\esp[N^\gamma(1)]<\infty$.
\end{lemma}

\begin{proof}
  For $k\geq2$, define $c_k = (k-1)^{-\gamma} - k^{-\gamma}$. Then,
  $\sum_{k=2}^\infty c_k = 1$ and applying summation by parts, we have
\begin{align*}
  \esp [ N^{-\gamma} (t) \mathbf1_{\{N(t)>0\}}] & = \sum_{k=1}^\infty
  k^{-\gamma} \pr(N(t)=k) = \sum_{k=1}^\infty k^{-\gamma} \{ \pr(N(t) \geq k) -  \pr(N(t)\geq k+1)\} \\
  & = \pr(N(t) \geq 1) - \sum_{k=2}^\infty c_k \pr(N(t) \geq k) \\
  & =  \pr(t_1 \leq t) - \sum_{k=2}^\infty c_k \pr(t_k \leq t) 
  = -  \pr(t_1>t) + \sum_{k=2}^\infty c_k  \pr(t_k>t) \; .
\end{align*}
Without loss of generality, assume that the intensity of the point process is
$\lambda=1$. Then, by definition of $c_k$, we have, for $t\geq2$, 
\begin{align*}
t^\gamma   \sum_{k \geq [t/2]+1} c_k \pr(t_k>t) \leq t^\gamma ([t/2])^{-\gamma} = O(1) \; .
\end{align*}
For $k \leq [t/2]$, we have, by Markov's inequality, 
\begin{align*}
  \pr(t_k >t) = \pr(t_k - k > t - k) \leq \pr(t_k - k > t/2) \leq c t^{-\gamma}
  \esp[|t_k-k|^\gamma]
\end{align*}
Applying the Ryll-Nardzewski inversion formula (\citet[Formula
1.2.25]{baccelli:bremaud:2003}), we have 
\begin{align*}
  \esp[|t_k-k|^\gamma] = \esp^0[t_1|t_k-k|^\gamma] \leq
  \{\esp^0[t_1^{1+\gamma}]\}^{1/(\gamma+1)} \{\esp^0[| t_k - k|^{\gamma+1}|]\}^{\gamma/(\gamma+1)} \; .
\end{align*}
Thus, applying Condition~(\ref{eq:bound-tk-palm}), we obtain that $\pr(t_k >t)
\leq c' t^{-\gamma} k^{\gamma\delta}$ and thus
\begin{align*}
  t^{\gamma} \sum_{2 \leq k \leq [t/2]} c_k \pr(t_k >t) \leq c' \sum_{2 \leq k
    \leq [t/2]} c_k k^{\gamma\delta} \leq c' \sum_{2 \leq k \leq [t/2]}
  k^{-\gamma(1-\delta)-1} = O(1) \; .
\end{align*}
This concludes the proof of~(\ref{eq:moment-cond}). We now consider the positive moments of $N(1)$. 
Applying summation by part, we have
\begin{align*}
  \esp[N^\gamma(1)] = \sum_{k=1}^\infty \{k^\gamma-(k-1)^\gamma\} \pr(N(1)\geq
  k) = \sum_{k=1}^\infty \{k^\gamma-(k-1)^\gamma\} \pr( t_k \leq 1) \; .
\end{align*}
For $k\geq2$ and $q>0$, we have, still assuming that $\lambda=1$, 
\begin{align*}
  \pr( t_k \leq 1) & \leq \pr(t_k - k \leq -k/2) \leq \esp[|t_k-k|^q] k^{-q} \; .
\end{align*}
Applying again the Ryll-Narzewski formula and
Condition~(\ref{eq:bound-tk-palm}), we obtain, for $k\geq2$,
\begin{align*}
  \pr( t_k \leq 1) \leq c k^{-q(1-\delta)} \; .
\end{align*}
Thus, 
\begin{align*}
  \esp[N^\gamma(1)] \leq 1 +  c \sum_{k=1}^\infty \{k^\gamma-(k-1)^\gamma\} k^{-q(1-\delta)} \; .
\end{align*}
The series is convergent as long as $q(1-\delta)>\gamma$. 
\end{proof}

\begin{proof}[Proof of~(\ref{eq:moment-cond-8}) for the LMSD model]
  Consider the LMSD model of Example~\ref{xmpl:LMSD}.  It is proved in
  \citet[Proposition~1]{MR2507532} that (\ref{eq:bound-tk-palm}) holds with
  $\delta = H_\tau$ if $\esp[\epsilon_0^p]<\infty$ for all $p\geq1$.  Actually,
  a close inspection of the first  lines of the proof shows that only $q$
  finite moments of $\epsilon_0$ are needed. Thus (\ref{eq:moment-cond-8}) holds if
  $E^0[\epsilon_0^{9-4H}]<\infty$, and $\esp[N^4(1)]<\infty$ if
  $E^0[\epsilon_0^q]<\infty$ for some $q>1+4/(1-H_\tau)$.
\end{proof}

\begin{proof}[Proof of~(\ref{eq:moment-cond-8}) for the ACD model]
  Under the assumptions of Example~\ref{xmpl:ACD}, the sequence $\{\tau_k\}$ is
  geometrically $\beta$-mixing \citet[Proposition~17]{MR1885348}. Denote
  $m=\esp^0[\tau_1]$.  The sequence $\{t_k\}$ is geometrically mixing, hence
  geometrically strong mixing. Thus, by \citet[Theorem~2.5]{rio:2000}, for
  $q\geq2$, if $E^0[\tau_1^{q+1+\epsilon}]<\infty$ for some $\epsilon>0$, then
  $E^0[|t_n-mn|^{q+1}] = O( n^{(q+1)/2})$. Thus~(\ref{eq:bound-tk-palm})
  holds with $\delta=1/2$.
\end{proof}

\end{document}